\documentclass[11pt,a4paper]{article}
\usepackage[margin=1in]{geometry}
\usepackage[numbers,sort&compress]{natbib}

\usepackage{amsmath,amssymb,amsfonts,amsthm,mathtools}
\usepackage{microtype}
\usepackage{mathrsfs}
\usepackage{bm}
\usepackage{booktabs}
\usepackage{array}
\usepackage{enumitem}
\usepackage{aliascnt}
\usepackage{xcolor}
\usepackage{hyperref}
\usepackage[nameinlink,capitalise,noabbrev]{cleveref}
\usepackage{bookmark}

\definecolor{linkblue}{RGB}{0,70,140}
\hypersetup{
  colorlinks=true,
  linkcolor=linkblue,
  citecolor=linkblue,
  urlcolor=linkblue,
  pdftitle={Strong well-posedness for cyclic catalytic reaction--diffusion SPDEs with mutation and immigration},
  pdfauthor={Jie Xiong and Wen Xu},
  pdfkeywords={cyclic catalytic reaction--diffusion SPDE, square-root coefficient,  weak variational solution, strong solution, pathwise uniqueness, Hellinger distance}
}

\allowdisplaybreaks
\numberwithin{equation}{section}

\theoremstyle{plain}
\newtheorem{theorem}{Theorem}[section]

\newaliascnt{proposition}{theorem}
\newtheorem{proposition}[proposition]{Proposition}
\aliascntresetthe{proposition}

\newaliascnt{lemma}{theorem}
\newtheorem{lemma}[lemma]{Lemma}
\aliascntresetthe{lemma}

\newaliascnt{corollary}{theorem}
\newtheorem{corollary}[corollary]{Corollary}
\aliascntresetthe{corollary}

\newaliascnt{assumption}{theorem}
\newtheorem{assumption}[assumption]{Assumption}
\aliascntresetthe{assumption}

\theoremstyle{definition}
\newaliascnt{definition}{theorem}
\newtheorem{definition}[definition]{Definition}
\aliascntresetthe{definition}

\newaliascnt{example}{theorem}

\aliascntresetthe{example}

\theoremstyle{remark}
\newaliascnt{remark}{theorem}
\newtheorem{remark}[remark]{Remark}
\aliascntresetthe{remark}

\crefname{theorem}{theorem}{theorems}
\Crefname{theorem}{Theorem}{Theorems}
\crefname{proposition}{proposition}{propositions}
\Crefname{proposition}{Proposition}{Propositions}
\crefname{lemma}{lemma}{lemmas}
\Crefname{lemma}{Lemma}{Lemmas}
\crefname{corollary}{corollary}{corollaries}
\Crefname{corollary}{Corollary}{Corollaries}
\crefname{assumption}{assumption}{assumptions}
\Crefname{assumption}{Assumption}{Assumptions}
\crefname{definition}{definition}{definitions}
\Crefname{definition}{Definition}{Definitions}
\crefname{remark}{remark}{remarks}
\Crefname{remark}{Remark}{Remarks}
\crefname{example}{example}{examples}
\Crefname{example}{Example}{Examples}

\newcommand{\R}{\mathbb{R}}
\newcommand{\N}{\mathbb{N}}
\newcommand{\E}{\mathbb{E}}
\newcommand{\Prob}{\mathbb{P}}
\newcommand{\one}{\mathbf{1}}
\newcommand{\dd}{\,\mathrm{d}}
\newcommand{\ip}[2]{\langle #1,#2\rangle}
\newcommand{\norm}[1]{\lVert #1\rVert}
\newcommand{\abs}[1]{\lvert #1\rvert}
\newcommand{\HS}{\operatorname{HS}}
\newcommand{\essinf}{\operatorname*{ess\,inf}}
\newcommand{\cF}{\mathcal{F}}
\newcommand{\cH}{\mathcal{H}}
\newcommand{\cX}{\mathcal{X}}
\newcommand{\cW}{\mathcal{W}}

\newcommand{\eps}{\varepsilon}

\newcommand{\Qmat}{\mathbf Q}
\newcommand{\Qop}{\mathsf Q}

\begin{document}

\title{Strong well-posedness for cyclic catalytic reaction--diffusion SPDEs with mutation and immigration}

\author{Jie Xiong\thanks{Department of Mathematics and Shenzhen International Center for Mathematics, Southern University of Science and Technology, Shenzhen 518055, China. Email: \texttt{xiongj@sustech.edu.cn}.}
\and Wen Xu\thanks{School of Mathematical Sciences, Peking University, Beijing 100871, China. Email: \texttt{xuwen@math.pku.edu.cn}.}}
\date{}

\maketitle

\begin{abstract}
\unboldmath
We study cyclic catalytic stochastic reaction–diffusion systems with mutation and immigration. The equations have square-root catalytic coefficients and spatially colored noise. Under a condition requiring the cyclic inward drift to dominate the corresponding noise covariance, we establish global strong well-posedness for nonnegative initial states. Our approach combines positivity-preserving approximations, uniform energy estimates, and a concave power entropy that controls the occupation of states near zero. These boundary estimates allow us to derive a Hellinger stability inequality for two synchronized weak variational solutions, yielding pathwise uniqueness directly in this solution class. We then use joint compactness of the approximating solutions and their driving noise, together with martingale identification and the Gyöngy–Krylov criterion, to construct a probabilistically strong solution on the prescribed Wiener space.  
\end{abstract}



\section{Introduction}\label{sec:intro}

In catalytic branching models, the branching rate of one population
depends on the density of another. In a cyclic system, population $i+1$
acts as the catalyst for population $i$, leading to noise coefficients
of the form $\sqrt{u_i(t,x)u_{i+1}(t,x)}$.
Mutually catalytic superprocesses were studied by Dawson and Perkins
and their collaborators
\cite{DawsonPerkins1998,DEFMPX1,DEFMPX2,DFMPX}, while cyclically catalytic super-Brownian motion was
introduced by Fleischmann and Xiong
\cite{FleischmannXiong2001}.

 At the finite-dimensional level, Dawson et. al \cite{DawsonFleischmannXiong2005} proved strong uniqueness for the spaceless cyclic system 
\[
 \dd X_i(t)=\alpha_iX_i(t)\,\dd t
 +\sqrt{\gamma_iX_i(t)X_{i+1}(t)}\,\dd B_i(t),
 \qquad i=1,\ldots,K.
\]
In the equation above, a coordinate that reaches zero is trapped there, and the proof relies crucially on this absorbing property.  By contrast, the off-diagonal mutation and immigration terms
can push a zero coordinate back into the interior, so that the boundary-trapping argument is unavailable.  Thus the theorem in 
\cite{DawsonFleischmannXiong2005} does not address the  competition between the inward term $q_{i,i+1}X_{i+1}$ and the variance $\lambda_iX_iX_{i+1}$.

He \cite{He2009}  subsequently studied strong uniqueness for finite-dimensional catalytic branching diffusions with diffusion coefficients of the form $\sqrt{X_i f_i(X)}$ under additional structural hypotheses .  More recently, Xiong and Xu \cite{XiongXuMutually2026} proved pathwise uniqueness for the two-dimensional spaceless mutually catalytic system with arbitrary mutation rates $q_{12},q_{21}\ge0$, positive branching parameters, and an interior initial state, without imposing a drift--covariance gap.   

Weak uniqueness for broader finite-dimensional catalytic branching networks was developed through martingale-problem and perturbation methods by Dawson and Perkins \cite{DawsonPerkins2006} and subsequently by Kliem \cite{Kliem2009}.  These results establish uniqueness in law for classes of singular diffusions, rather than the pathwise Hellinger stability or the prescribed-noise strong construction proved here.  They nevertheless clarify why weak uniqueness cannot simply be taken for granted when a catalytic coefficient degenerates on a union of coordinate faces.

For scalar stochastic heat equations with non-Lipschitz multiplicative coefficients, the available uniqueness mechanisms depend strongly on the spatial covariance.  Mytnik et. al  \cite{MytnikPerkinsSturm2006} treated sufficiently regular spatially colored noise by a Yamada--Watanabe argument; for space--time white noise, Mytnik and Perkins obtained pathwise uniqueness under a stronger H\"older threshold in the state variable \cite{MytnikPerkins2011}.  Neither result directly controls the vector coefficient $\sqrt{u_i u_{i+1}}$, whose singular derivatives couple two coordinates and whose useful cancellation also involves the incoming cyclic drift.

A different strong-solution route represents a measure-valued process through its cumulative distribution function.  Xiong obtained a distribution-function-valued SPDE representation of super-Brownian motion and proved strong uniqueness by an extended Yamada--Watanabe method \cite{Xiong2013}; related distribution-function-valued systems and interacting superprocesses were studied by Wang et. al \cite{WangYangZhou2017} and by Ji et. al \cite{JiLiuXiong2022}.  In those works the state variable is a distribution function and the driving white noise is indexed through the cumulative mass variable.  The present state is instead an $L^2$-density vector on a bounded spatial interval, and its trace-class colored covariance acts through fixed spatial modes.  Consequently, the distribution-function transformation and its order structure are not available here.

The stochastic partial differential equation studied here is
\begin{equation}\label{eq:intro-spde}
\begin{aligned}
 \dd u_i(t,x)
 &=\left[\frac12\Delta_Nu_i(t,x)+\sum_{j=1}^Kq_{ij}(x)u_j(t,x)+\theta_i(x)\right]\dd t\\
 &\quad+\sqrt{u_i(t,x)u_{i+1}(t,x)}
       \sum_{m=1}^{\infty}\sigma_{im}(x)\,\dd\beta_{im}(t),
 \qquad i=1,\ldots,K.
\end{aligned}
\end{equation}
The coefficient field $\Qmat=(q_{ij})$ is a \emph{Metzler reaction matrix}: all off-diagonal entries are nonnegative and represent the mutation rates, while the diagonal entries may have either sign.  The terminology allows mutation or transfer between types, but no conservativity or Markov-generator condition is imposed, diagonal growth and killing are permitted.  The vector $\theta$ is nonnegative and represents immigration. The processes $\beta_{im},\ i=1,2,\cdots,K;\ m=1,2,\cdots$ are independent Brownian motions. We use the convention $u_{K+1}=u_1$ throughout, the same convention applies to $q_{i,i+1}$, so $q_{K,K+1}$ means $q_{K,1}$.

The proof separates two approximation procedures.  For a fixed boundary regularization, an auxiliary amplitude truncation gives a globally Lipschitz Hilbert--Schmidt coefficient.  The corresponding finite-mode equations preserve the positive cone, and their covariance-tail estimate permits $M\to\infty$, producing the equation driven by the complete cylindrical Wiener process.  The amplitude truncation is then removed by consistency and a nonexplosion estimate.  Only after this construction is complete do we let the boundary scale tend to zero.

The proof of the principal theorem has four stages.  First, we construct the countably driven boundary-regularized equations and prove positivity, global variational bounds, and continuous $C(\overline I)^K$-valued paths.  Second, a concave power entropy uses the strict cyclic drift--covariance gap to control the weighted occupation of the coordinate faces.  Third, we apply a two-component variational It\^o formula directly to synchronized nonnegative weak solutions.  The heat term is dissipative by convexity, the common-noise trace is written with its cross-covariation term, and the remaining boundary error is removed by the power entropy.  This gives pathwise uniqueness in the probabilistically weak variational class.  Fourth, we identify every joint state--noise limit by its martingales, quadratic variations, and cross-variations.  The Gy\"ongy--Krylov criterion then yields convergence on the original stochastic basis, after which adaptedness to the prescribed noise and the full stochastic integral are established.

The relation to the authors' finite-dimensional  paper \cite{XiongXu2026} is as follows.  What is inherited is the scalar boundary linearization $g_n(r)=r/\sqrt{r+\varepsilon_n}$ and the recognition, from boundary-sensitive Malliavin estimates, that the combination $q_{i,i+1}-\gamma_i/4$ is structurally relevant for that SDE approximation.  That paper constructs strong solutions through Malliavin compactness and obtains pathwise uniqueness by combining strong existence with weak uniqueness.  By contrast, it contains neither the two-solution Hellinger identity  nor a direct synchronized-solution stability estimate.  Malliavin compactness and the dual Yamada--Watanabe route are not used to prove the SPDE uniqueness theorem below.

The main contributions can be matched to the three main theorems.  First, \Cref{thm:power-main} adds the heat-energy production term, proves boundary occupation for every incoming Metzler component, and does so for the complete spatial covariance diagonal.  Second, \Cref{thm:stability} is a direct pathwise-stability theorem for two arbitrary synchronized weak variational solutions: its proof combines the Hessian dissipation of the Neumann heat operator, an infinite-mode mixed trace, and a two-component variational It\^o formula.  Third, \Cref{thm:main} constructs the solution on a prescribed Wiener space by keeping the state and the entire countable Brownian family in the same compactness problem and identifying martingales, brackets, and cross-variations in the limiting joint filtration.  Thus the novelty is the combination of Hellinger geometry, parabolic dissipation, spatially dependent trace-class color, a general Metzler drift, and complete state--noise identification.  Stochastic factorization, the variational It\^o formula, and the Gy\"ongy--Krylov criterion are used as the main tools.

The paper is organized as follows.  \Cref{sec:framework} introduces the reaction matrix, the cylindrical coordinate process, weak and strong solution concepts, and the main results.  \Cref{sec:approx} constructs the countably driven regularized equations.  \Cref{sec:apriori} proves energy, temporal, spatial, and boundary estimates.  \Cref{sec:hellinger-algebra} establishes Hellinger stability and pathwise uniqueness for synchronized weak solutions.  \Cref{sec:compactness} identifies all weak limits and constructs the strong solution on the prescribed Wiener space.

\section{Preliminaries and main results}\label{sec:framework}

Let $I=(0,1)$, $\bar I=[0,1]$, and define
$
H:=L^2(I), V:=H^1(I), V^*:=(H^1(I))^*,
$ 
so that $V\subset H\subset V^*$ is a Gelfand triple.  For a real-valued function $r$ write $r^{+}=\max\{r,0\}$ and $r^{-}=\max\{-r,0\}$.  
We use $\ip{\cdot}{\cdot}$ for the inner product in $H$
and for duality pairings, with subscripts when needed.
Let $A=\frac12\Delta_N$, where $\Delta_N$ is the Neumann
 Laplacian on $H$. 
 
  In variational form,
 \begin{equation}\label{eq:Avar}
 	\ip{Au}{\varphi}_{V^*,V}
 	=-\frac12\int_I\partial_xu\,\partial_x\varphi\,\dd x,
 	\qquad u,\varphi\in V.
 \end{equation}
 Let $P_t=e^{tA}$ denote the analytic semigroup generated
 by $A$. 
  We equip $H^K$ and $V^K$ with the norms
  $\|u\|_{H^K}^2=\sum_{i=1}^K\|u_i\|_H^2$ and
  $\|u\|_{V^K}^2=\sum_{i=1}^K\|u_i\|_V^2$, respectively. We write $D(B)$ for the domain of an operator $B$ and
  $\mathrm{Id}$ for the identity operator on $H$.

Let $(e_k)_{k\ge0}$ be the real orthonormal Neumann eigenbasis
of $H$, with $-Ae_k=\nu_ke_k$, where
$e_0=1$, $e_k(x)=\sqrt{2}\cos(k\pi x)$ for $k\ge1$,
and $\nu_k=k^2\pi^2/2$. For $\alpha\ge0$, define
\begin{equation}\label{eq:Halpha-definition}
	H_A^\alpha(I):=D\bigl((\mathrm{Id}-A)^{\alpha/2}\bigr),
	\qquad
	\|v\|_{H_A^\alpha}^2
	=\sum_{k=0}^{\infty}(1+\nu_k)^\alpha
	|\ip{v}{e_k}_H|^2.
\end{equation}
For $0\le\alpha\le1$, this space coincides with the usual
Sobolev space $H^\alpha(I)$ with equivalent norms, and we
omit the subscript $A$. In particular,
$H^\alpha(I)\hookrightarrow C(\bar I)$ continuously
for $1/2<\alpha\le1$.
For $\alpha<0$, let $H_A^\alpha(I)$ be the completion of $H$
under the spectral norm in \eqref{eq:Halpha-definition}.
We write
\[
H^{-2}(I):=H_A^{-2}(I)=D(A)^*,
\qquad
D(A)=H_A^2(I)
=\{v\in H^2(I):v'(0)=v'(1)=0\},
\]
 In particular, its dual inner product is given by 
\[
\ip{f}{\varphi}_{H^{-2},D(A)}
=\ip{f}{\varphi}_H,
\qquad f\in H,\quad \varphi\in D(A).
\]

We next introduce the driving cylindrical Wiener processes.
For each $i=1,\ldots,K$, let
\[
U_i=\ell^2
:=\left\{a=(a_m)_{m\ge1}:\ a_m\in\mathbb R,\ 
\sum_{m=1}^{\infty}a_m^2<\infty\right\},
\qquad
\langle a,b\rangle_{U_i}
=\sum_{m=1}^{\infty}a_mb_m.
\]
Let $(f_m)_{m\ge1}$ be its canonical orthonormal basis (ONB) of $\ell^2$.
Write
\begin{equation}\label{eq:1013}
W_i(t)=\sum_{m=1}^{\infty}\beta_{im}(t)f_m.
\end{equation}
The processes $W_1,\ldots,W_K$ are thus independent
cylindrical Wiener processes on $U_1,\ldots,U_K$, respectively.

For each $i=1,\ldots,K$, let $(\sigma_{im})_{m\ge1}\subset H$
be deterministic spatial functions such that
\begin{equation*}
\sum_{m=1}^{\infty}\|\sigma_{im}\|_H^2<\infty.
\end{equation*}

Define the following  covariance kernel and its diagonal 
\begin{equation}\label{eq:covkernel}
 K_i(x,y)=\sum_{m=1}^{\infty}\sigma_{im}(x)\sigma_{im}(y),
 \qquad
 \lambda_i(x)=K_i(x,x)=\sum_{m=1}^{\infty}\sigma_{im}(x)^2.
\end{equation}
For $N\ge1$, write

\begin{equation}\label{eq:lambda-tail}
 \lambda_i^{>N}(x)=\sum_{m>N}\sigma_{im}(x)^2,
 \qquad
 \rho_N=\max_{1\le i\le K}\|\lambda_i^{>N}\|_{L^\infty(I)}.
\end{equation}

Next, we define the Polish path space for the infinite Brownian family.
Fix $s_0>1/2$ and define
\[
 U_{i,0}=\left\{z=(z_m)_{m\ge1}:
 \|z\|_{U_{i,0}}^2=\sum_{m=1}^{\infty}(1+m)^{-2s_0}|z_m|^2<\infty\right\}.
\]
The embedding $U_i\hookrightarrow U_{i,0}$ is Hilbert--Schmidt.  Hence $W_i$ has a version in $C([0,T];U_{i,0})$. 
Define 
\[
 \mathbb U=\bigoplus_{i=1}^KU_i,
 \qquad
 \mathbb U_0=\bigoplus_{i=1}^KU_{i,0},
 \qquad
 W=(W_1,\ldots,W_K).
\]
The noise path $W$ is therefore a random variable in the Polish space $C([0,T];\mathbb U_0)$.  

Now we state the assumptions needed in this article.

\begin{assumption}\label{ass:coefficients}
Let $K\ge2$.  For all $i,j=1,\ldots,K$,
\[
 q_{ij}\in L^\infty(I),
 \quad q_{ij}(x)\ge0\quad\text{for }i\ne j\text{ and a.e. }x,
\]
  Moreover,
\[
 \theta_i\in L^\infty(I),\qquad \theta_i\ge0\quad\text{a.e.},
\]
and
\begin{equation}\label{eq:mode-summability}
 \sum_{m=1}^{\infty}\|\sigma_{im}\|_{L^\infty(I)}^2<\infty.
\end{equation}
The deterministic initial condition satisfies
\[
 u_0=(u_{1,0},\ldots,u_{K,0})\in V^K,
 \qquad u_{i,0}\ge0\quad\text{a.e.}
\]
\end{assumption}

We shall use the column-growth constant
\begin{equation}\label{eq:q-constants}
 q_{\mathrm H}
 =\max_{1\le j\le K}
 \left\|2q_{jj}^{+}+\sum_{i\ne j}q_{ij}\right\|_{L^\infty(I)}.
\end{equation}
 We also write

\begin{equation}\label{eq:coefficient-norms}
\mathfrak q_0=\max_{i,j}\|q_{ij}\|_{L^\infty},\qquad
\mathfrak t_0=\max_i\|\theta_i\|_{L^\infty},\qquad
\mathfrak s_0=\max_i\sum_{m\ge1}\|\sigma_{im}\|_{L^\infty}^2,
\end{equation}
and
\begin{equation}\label{eq:Lambda}
 \Lambda=\max_{1\le i\le K}\|\lambda_i\|_{L^\infty(I)}.
\end{equation}
\begin{assumption}\label{ass:strict-gap}
There is $\delta_*>0$ such that
\begin{equation}\label{eq:strict-gap}
 q_{i,i+1}(x)-\frac14\lambda_i(x)\ge\delta_*
 \qquad\text{for every }i\text{ and a.e.  }x\in I.
\end{equation}
\end{assumption}

Unless a dependence is displayed explicitly, a generic constant may depend on $K$, the time horizon, $\mathfrak q_0$, $\mathfrak t_0$, $\mathfrak s_0$, $\delta_*$, and the indicated moment of $u_0$, but never on the boundary index $n$, the mode cutoff $M$, the amplitude cutoff $L$, or the entropy regularization $\eta$.  A constant carrying a subscript $R$, $L$, or $n$ may additionally depend on that parameter.

We next define the state-dependent noise operators and the SPDE.
For a nonnegative vector $z=(z_1,\ldots,z_K)$, define
\begin{equation}\label{eq:G-operator}
 G_i(z)\in L_2(U_i,H),
 \qquad
 G_i(z)f_m=\sigma_{im}\sqrt{z_i z_{i+1}}.
\end{equation}
Then
\begin{equation}\label{eq:G-HS}
 \|G_i(z)\|_{L_2(U_i,H)}^2
 =\int_I\lambda_i(x)z_i(x)z_{i+1}(x)\,\dd x
 \le\frac{\|\lambda_i\|_\infty}{2}
 \bigl(\|z_i\|_H^2+\|z_{i+1}\|_H^2\bigr).
\end{equation}
The equation is rewritten as
\begin{equation}\label{eq:spde}
\left\{
\begin{aligned}
 \dd u_i(t)
 &=\bigl[Au_i(t)+(\Qop u(t))_i+\theta_i\bigr]\dd t
   +G_i(u(t))\,\dd W_i(t),\\
 u_i(0)&=u_{i,0},
\end{aligned}
\right.
\qquad i=1,\ldots,K.
\end{equation}

\begin{definition}\label{def:weak-solution}
Fix $T>0$.  A weak variational (WV) solution of \eqref{eq:spde} consists of a stochastic basis
$(\Omega,\cF,(\cF_t)_{t\in[0,T]},\Prob)$ satisfying the usual conditions, independent $(\cF_t)$-cylindrical Wiener processes $W_i$ on $U_i$, and an $(\cF_t)$-progressively measurable nonnegative process $u$ such that
\begin{equation}\label{eq:solution-space}
 u\in L^2\bigl(\Omega;C([0,T];H^K)\bigr)
 \cap L^2\bigl(\Omega\times[0,T];V^K\bigr),
 \qquad u_i\ge0,
\end{equation}
and, for every $\varphi\in V$, every $i$, and every $t\in[0,T]$,
\begin{align}\label{eq:variational}
 \ip{u_i(t)}{\varphi}
 &=\ip{u_{i,0}}{\varphi}
 -\frac12\int_0^t\!\int_I\partial_xu_i(s)\partial_x\varphi\,\dd x\dd s
 +\int_0^t\ip{\sum_{j=1}^Kq_{ij}u_j(s)+\theta_i}{\varphi}\dd s\notag\\
 &\quad+\int_0^t\ip{G_i(u(s))^*\varphi}{\dd W_i(s)}_{U_i}
\end{align}
almost surely.   Two WV solutions, possibly corresponding to different admissible initial states, are \emph{synchronized} when they are defined on the same stochastic basis and are driven by the same cylindrical Wiener processes.  
\end{definition}

\begin{definition}\label{def:strong-solution}
On a prescribed stochastic basis carrying $W$, a WV solution is probabilistically strong if it is adapted to the completed natural filtration $(\cF_t^W)$ of $W$.  Pathwise uniqueness means that any two synchronized WV solutions with the same initial state are indistinguishable.  
\end{definition}

\begin{remark}[Mild formulation]\label{rem:mild}
Once \eqref{eq:variational} holds, semigroup duality gives
\begin{align}\label{eq:mild}
 u_i(t)
 &=P_tu_{i,0}
 +\int_0^tP_{t-s}\left[\sum_{j=1}^Kq_{ij}u_j(s)+\theta_i\right]\dd s
 +\int_0^tP_{t-s}G_i(u(s))\,\dd W_i(s).
\end{align}
\end{remark}

Finally, we state the principal results.  For $r,s\ge0$, set
\begin{equation}\label{eq:hellinger-distance}
 h(r,s)=\bigl(\sqrt r-\sqrt s\bigr)^2,
 \qquad
 \cH(u,v)=\sum_{i=1}^K\int_Ih(u_i(x),v_i(x))\,\dd x.
\end{equation}
For each row define the off-cycle incoming part
\begin{equation}\label{eq:R-i}
\mathcal R_i(z)=\sum_{\substack{1\le j\le K\\j\ne i,\,j\ne i+1}}q_{ij}z_j,
\end{equation}
where an empty sum is zero and the cyclic convention is understood. 

We now present the main results of this paper.

\begin{theorem}[Power entropy and boundary occupation]\label{thm:power-main}
Under \Cref{ass:coefficients,ass:strict-gap}, there exists $p\in(0,1/2)$ such that, for every $T>0$, every nonnegative WV solution of \eqref{eq:spde}, every $0<\eta\le1$, and every $i$,
\begin{align}\label{eq:power-main}
 &\E\int_0^T\!\int_I
 \frac{|\partial_xu_i|^2}{(u_i+\eta)^{2-p}}\,\dd x\dd t
 +\E\int_0^T\!\int_I
 \frac{u_{i+1}}{(u_i+\eta)^{1-p}}\,\dd x\dd t\notag\\
 &\quad+\E\int_0^T\!\int_I
 \frac{\mathcal R_i(u)+\theta_i}{(u_i+\eta)^{1-p}}\,\dd x\dd t
 \le C_T.
\end{align}
Consequently, for $0<r\le1$,
\begin{align}
 \E\int_0^T\!\int_Iu_{i+1}\one_{\{u_i\le r\}}\,\dd x\dd t
 &\le C_Tr^{1-p},\label{eq:occupation-main}\\
 \E\int_0^T\!\int_I\bigl(\mathcal R_i(u)+\theta_i\bigr)
 \one_{\{u_i\le r\}}\,\dd x\dd t
 &\le C_Tr^{1-p},\label{eq:occupation-other-main}\\
 \E\int_0^T\!\int_I|\partial_xu_i|^2\one_{\{u_i\le r\}}\,\dd x\dd t
 &\le C_Tr^{2-p}.\label{eq:occupation-gradient-main}
\end{align}
\end{theorem}

\begin{theorem}[Hellinger stability and pathwise uniqueness]\label{thm:stability}
Under \Cref{ass:coefficients,ass:strict-gap}, let $u$ and $v$ be synchronized nonnegative WV solutions of \eqref{eq:spde}.  Then
\begin{equation}\label{eq:stability-main}
 \E\cH(u(t),v(t))
 \le e^{\mathfrak q_{\mathrm H}t}\E\cH(u(0),v(0)),
 \qquad 0\le t\le T.
\end{equation}
In particular, pathwise uniqueness holds.
\end{theorem}

\begin{theorem}[Strong well-posedness]\label{thm:main}
Under \Cref{ass:coefficients,ass:strict-gap}, equation \eqref{eq:spde} has a unique nonnegative probabilistically strong WV solution driven by the prescribed cylindrical Wiener process $W$.  It satisfies \eqref{eq:mild} and has a version in
\[
 C\bigl([0,T];C(\bar I)^K\bigr)
 \quad\text{for every }T>0.
\]
Furthermore, for every $T>0$ and $q\ge1$,
\begin{equation}\label{eq:main-energy}
 \E\sup_{t\le T}\|u(t)\|_{H^K}^{2q}
 +\E\left(\int_0^T\|u(t)\|_{V^K}^2\,\dd t\right)^q
 \le C_{q,T}\bigl(1+\|u_0\|_{H^K}^{2q}\bigr).
\end{equation}
\end{theorem}

\section{Boundary-linearized approximations}\label{sec:approx}

This section constructs the regularized equations.  For $n\in\N$, let $\eps_n=2^{-n}$,  and for $r\ge0$, define
\begin{equation}\label{eq:gn}
 g_n(r)=\frac{r}{\sqrt{r+\eps_n}},
 \qquad
 p_n(r)=g_n(r)^2=\frac{r^2}{r+\eps_n}.
\end{equation}

The proof of the following lemma is elementary so we omit it.
\begin{lemma}\label{lem:g-properties}
For $r,s,a,b\ge0$,  the following estimates hold:
\begin{align}
 0\le p_n(r)&\le r, \qquad 0\le p'_n(r)\le1,\label{eq:g1}\\
 |g_n(r)-\sqrt r|^2&\le\eps_n,  \qquad
 |g_n(r)-g_n(s)|^2\le|r-s|,\label{eq:g2}\\
 |g_n(a)g_n(b)-\sqrt{ab}|^2&\le2\eps_n(a+b).\label{eq:product-approx}
\end{align}
\end{lemma}

Let $u_0^n=u_0+n^{-1}\mathbf1$.  For nonnegative $z$, define
\begin{equation}\label{eq:Gn-operator}
 G_i^n(z)f_m=\sigma_{im}g_n(z_i)g_n(z_{i+1}).
\end{equation}
Then
\begin{equation}\label{eq:Gn-HS}
 \|G_i^n(z)\|_{L_2(U_i,H)}^2
 =\int_I\lambda_i p_n(z_i)p_n(z_{i+1})\,\dd x
 \le\frac{\|\lambda_i\|_\infty}{2}
 \bigl(\|z_i\|_H^2+\|z_{i+1}\|_H^2\bigr).
\end{equation}
The boundary-regularized equation is
\begin{equation}\label{eq:approx}
\left\{
\begin{aligned}
 \dd u_i^n(t)
 &=\left[Au_i^n(t)+\sum_{j=1}^Kq_{ij}u_j^n(t)+\theta_i\right]\dd t
   +G_i^n(u^n(t))\,\dd W_i(t),\\
 u_i^n(0)&=u_{i,0}+n^{-1}.
\end{aligned}
\right.
\end{equation}

For $L\ge1$, let $\chi_L(r)=(r^+)\wedge L$ and set
\[
 g_{n,L}(r)=g_n(\chi_L(r)).
\]
Let $\Pi_M$ be the projection of $U_i$ onto
$\operatorname{span}\{f_1,\ldots,f_M\}$, with $\Pi_\infty=I$, and define
\begin{equation}\label{eq:GnLM}
 G_i^{n,L,M}(z)f_m
 =\one_{\{m\le M\}}\sigma_{im}g_{n,L}(z_i)g_{n,L}(z_{i+1}).
\end{equation}

The truncated coefficients satisfy the following estimates.

\begin{lemma}\label{lem:Gn-truncated}
For fixed $n\in\N$ and $L\ge1$, uniformly in $M\in\N\cup\{\infty\}$,
\begin{equation}\label{eq:Gn-global-Lip}
 \sum_{i=1}^K\|G_i^{n,L,M}(z)-G_i^{n,L,M}(w)\|_{L_2(U_i,H)}^2
 \le C_{n,L}\|z-w\|_{H^K}^2,
 \qquad z,w\in H^K.
\end{equation}
For every $z\in H^K$,
\begin{equation}\label{eq:GnLM-growth}
 \sum_i\|G_i^{n,L,M}(z)\|_{L_2(U_i,H)}^2\le C\|z\|_{H^K}^2,
\end{equation}
and
\begin{equation}\label{eq:Gn-tail}
 \sum_i\|G_i^{n,L,\infty}(z)(I-\Pi_M)\|_{L_2(U_i,H)}^2
 \le C\rho_M\|z\|_{H^K}^2.
\end{equation}
The constants in \eqref{eq:GnLM-growth}--\eqref{eq:Gn-tail} are independent of $n,L,M$.
\end{lemma}

\begin{proof}
Define 
\[
 B_{n,L}=\|g_{n,L}\|_{L^\infty(\R)}\le\sqrt L,
 \qquad \ell_{n,L}=\operatorname{Lip}(g_{n,L})<\infty.
\]
For every $a,b,c,d\in\R$, it is easy to see that
\begin{align*}
 |g_{n,L}(a)g_{n,L}(b)-g_{n,L}(c)g_{n,L}(d)|^2
 &\le2B_{n,L}^2\ell_{n,L}^2
 \bigl(|a-c|^2+|b-d|^2\bigr).
\end{align*}
Consequently,
\begin{align*}
 &\|G_i^{n,L,M}(z)-G_i^{n,L,M}(w)\|_{L_2(U_i,H)}^2\\
 &\quad=\int_I\sum_{m\le M}\sigma_{im}^2
 |g_{n,L}(z_i)g_{n,L}(z_{i+1})
   -g_{n,L}(w_i)g_{n,L}(w_{i+1})|^2\,\dd x\\
 &\quad\le2\|\lambda_i\|_\infty B_{n,L}^2\ell_{n,L}^2
 \bigl(\|z_i-w_i\|_H^2+\|z_{i+1}-w_{i+1}\|_H^2\bigr).
\end{align*}
Summing over $i$ proves \eqref{eq:Gn-global-Lip}, uniformly in $M$.

For arbitrary $z\in H^K$,
$g_{n,L}(z_k)^2=p_n(\chi_L(z_k))\le \chi_L(z_k)\le z_k^+\le |z_k|$, and hence
\[
 \|G_i^{n,L,M}(z)\|_{L_2(U_i,H)}^2
 \le\int_I\lambda_i |z_i z_{i+1}|\,\dd x
 \le\frac{\|\lambda_i\|_\infty}{2}
 \bigl(\|z_i\|_H^2+\|z_{i+1}\|_H^2\bigr).
\]
The same calculation with $\lambda_i$ replaced by
$\lambda_i^{>M}$ proves \eqref{eq:Gn-tail}.  Since
$\|\lambda_i^{>M}\|_\infty\le\rho_M$, the constants in the last two estimates are independent of $n,L,M$.
\end{proof}

The following proposition establishes the existence and uniqueness of nonnegative regularized solutions and justifies the removal of the mode cutoff.

\begin{proposition} \label{prop:approx-existence}
For every $n\in\N$, $L\ge1$, and $M\in\N\cup\{\infty\}$, the equation with coefficient $G^{n,L,M}$ has a unique global mild variational solution $u^{n,L,M}$.  It is nonnegative and satisfies, for every $T>0$ and $q\ge1$,
\begin{align}
 &\sup_{n,L,M}\left[
 \E\sup_{t\le T}\|u^{n,L,M}(t)\|_{H^K}^{2q}
 +\E\left(\int_0^T\|u^{n,L,M}(t)\|_{V^K}^2\,\dd t\right)^q
 \right]<\infty,\label{eq:uniform-nLM-energy}\\
 &\sup_{n,L,M}\E\|u^{n,L,M}\|_{C([0,T];C(\overline I)^K)}^q<\infty.\label{eq:uniform-nLM-C}
\end{align}
Moreover, for fixed $n\in\N$, $L\ge1$ and $M\in\N$, the  following holds
\begin{equation}\label{eq:mode-truncation-convergence}
 \E\sup_{t\le T}\|u^{n,L,M}(t)-u^{n,L,\infty}(t)\|_{H^K}^2
 +\E\int_0^T\|u^{n,L,M}-u^{n,L,\infty}\|_{V^K}^2\,\dd t
 \le C_{n,L,T}\rho_M.
\end{equation}
\end{proposition}

\begin{proof}
\emph{Step 1: globally Lipschitz equations.}
Let $\mathbb H=H^K$, $\mathbb V=V^K$, and let
$\mathbb A=\operatorname{diag}(A,\ldots,A)$. For fixed $n\in\N$, $L\ge1$, and $M\in\N\cup\{\infty\}$, the map
$G^{n,L,M}:\mathbb H\to L_2(\mathbb U,\mathbb H)$ is globally Lipschitz by
\eqref{eq:Gn-global-Lip}, while $\Qop$ is bounded on $\mathbb H$ and
\[
 2\,{}_{\mathbb V^*}\!\ip{\mathbb A z}{z}_{\mathbb V}
 =-\sum_{i=1}^K\|\partial_xz_i\|_H^2.
\]
The standard fixed-point theorem for stochastic evolution equations therefore gives a unique global mild $\mathbb H$-valued solution; see \cite[Chapter~7]{DaPratoZabczyk2014}.  Applying the Hilbert-space energy formula to its stopping-time localizations and using \eqref{eq:GnLM-growth} gives
\[
 \E\sup_{t\le T}\|u^{n,L,M}(t)\|_{\mathbb H}^2
 +\E\int_0^T\|u^{n,L,M}(t)\|_{\mathbb V}^2\,\dd t<\infty.
\]
The variational continuity theorem \cite[Chapter~4]{PrevotRockner2007} then identifies the mild solution with the unique variational solution.  This argument is valid also for $M=\infty$, because \eqref{eq:Gn-global-Lip} is an estimate in the full Hilbert--Schmidt space.

\emph{Step 2: invariance of the positive cone.}

The definition through $\chi_L$ already extends $g_{n,L}$ by zero on $(-\infty,0]$.  Choose convex functions $\phi_\delta\in C^2(\R)$, vanishing on $[0,\infty)$, such that, as $\delta\downarrow0$,
\[
 \begin{gathered}
 0\le\phi_\delta(r)\uparrow(r^-)^2,
 \qquad \phi_\delta'(r)\le0,
 \qquad |\phi_\delta'(r)|\le2r^-,\\
 0\le\phi_\delta''(r)\le2,
 \qquad \phi_\delta''(r)=0\ \text{for }r\ge0.
 \end{gathered}
\]
For fixed $\delta$ and $N$, choose a convex $\phi_{\delta,N}\in C^2(\R)$ which agrees with $\phi_\delta$ on $[-N,\infty)$, is affine on $(-\infty,-N-1]$, has uniformly continuous second derivative, and satisfies
\[
 \begin{gathered}
 0\le\phi_{\delta,N}(r)\le(r^-)^2,
 \qquad |\phi_{\delta,N}'(r)|\le2r^-,\\
 0\le\phi_{\delta,N}''(r)\le2,
 \qquad \phi_{\delta,N}''(r)=0\ \text{for }r\ge0.
 \end{gathered}
\]
Such functions are obtained by cutting off the curvature $\phi_\delta''$ smoothly to the left of $-N$ and integrating twice from the origin.  For each fixed $N$, the first two derivatives are bounded, so the variational It\^o formula applies directly to
$\sum_i\int_I\phi_{\delta,N}(u_i^{n,L,M})\,\dd x$.  The stochastic integral and its trace are identically zero, because $\phi_{\delta,N}'$ and $\phi_{\delta,N}''$ vanish on $[0,\infty)$ whereas the $i$th noise row vanishes when its $i$th state coordinate is negative.  Letting $N\to\infty$ is justified by the displayed bounds, the Step~1 energy estimate, dominated convergence for the endpoint and lower-order terms, and dominated convergence for the heat term against $|\partial_xu_i^{n,L,M}|^2$.  We thereby obtain the identity with $\phi_\delta$ used below.

The Neumann heat contribution is
\[
 -\frac12\sum_i\int_I\phi_\delta''(u_i^{n,L,M})
 |\partial_xu_i^{n,L,M}|^2\,\dd x\le0.
\]
The complete trace contribution is zero.  Indeed, on
$\{\phi_\delta''(u_i^{n,L,M})\ne0\}$ one has $u_i^{n,L,M}<0$, so
$g_{n,L}(u_i^{n,L,M})=0$ and the entire $i$th noise row vanishes, for every mode.  Since $\theta_i\ge0$ and $\phi_\delta'\le0$, immigration is nonpositive.  For $i\ne j$, the Metzler condition and
(suppressing the superscripts $n,L,M$ in the following two pointwise bounds)
$u_j=u_j^+-u_j^-$ give
\[
 \phi_\delta'(u_i)q_{ij}u_j
 \le2q_{ij}u_i^-u_j^-
 \le q_{ij}\bigl((u_i^-)^2+(u_j^-)^2\bigr),
\]
whereas
$\phi_\delta'(u_i)q_{ii}u_i\le2q_{ii}^+(u_i^-)^2$.
Taking expectations, sending $\delta\downarrow0$, and using dominated convergence in the bounded drift terms yields
\[
 \sum_i\E\|(u_i^{n,L,M}(t))^-\|_H^2
 \le C\int_0^t\sum_i\E\|(u_i^{n,L,M}(s))^-\|_H^2\,\dd s.
\]
The initial state is nonnegative, so Gronwall's lemma proves positivity.  

\emph{Step 3: cutoff-independent high moments and continuous spatial paths.}
Write
\[
 Y(t)=1+\|u^{n,L,M}(t)\|_{H^K}^2,
 \qquad D(t)=\sum_i\|\partial_xu_i^{n,L,M}(t)\|_H^2.
\]
The matrix, immigration, and trace bounds similar to those used in
\eqref{eq:matrix-energy-bound}--\eqref{eq:trace-energy-bound} below give, after localization,
\begin{equation}\label{eq:nLM-basic-energy}
 \dd Y(t)+D(t)\,\dd t\le C Y(t)\,\dd t+\dd M(t),
 \qquad \dd\langle M\rangle_t\le C Y(t)^2\,\dd t,
\end{equation}
with $C$ independent of $n,L,M$.  For every $q\ge1$, It\^o's formula for $Y^q$, followed by BDG and Young, gives
\begin{equation}\label{eq:nLM-Yq}
 \E\sup_{t\le T}Y(t)^q
 +\E\int_0^T Y(t)^{q-1}D(t)\,\dd t
 \le C_{q,T}Y(0)^q.
\end{equation}
Integrating \eqref{eq:nLM-basic-energy}, raising to the $q$th power, and using
\[
 \E\sup_{t\le T}|M(t)|^q
 \le C_q\E\left(\int_0^TY(s)^2\,\dd s\right)^{q/2}
 \le C_{q,T}\E\sup_{s\le T}Y(s)^q
\]
shows
\[
 \E\left(\int_0^TD(s)\,\dd s\right)^q\le C_{q,T}Y(0)^q.
\]
Since $\int_0^T\|u\|_V^2\le\int_0^T(Y+D)$, this proves
\eqref{eq:uniform-nLM-energy} with all constants independent of the cutoffs.

For the $C(\overline I)$-valued paths, use the mild equation.  For
$\alpha\in(1/2,1)$,
\begin{equation}\label{eq:semigroup-Halpha}
 \|(I-A)^{\alpha/2}P_rB\|_{L_2(U,H)}^2
 \le C_\alpha(1+r^{-\alpha})\|B\|_{L_2(U,H)}^2
,\qquad r>0.
\end{equation}
On a fixed horizon $0<r\le T$, this is bounded by $C_{\alpha,T}r^{-\alpha}\|B\|_{L_2(U,H)}^2$, which is the form used in the convolution estimates below.
Choose $q>2/(1-\alpha)$.  The stochastic factorization 
\Cref{lem:factorization}, \eqref{eq:GnLM-growth}, and
\eqref{eq:nLM-Yq} give a continuous $H^\alpha$-valued stochastic convolution and the corresponding cutoff-independent moment bound.  For an arbitrary target moment $r\ge1$, one first chooses such a $q\ge r$ and then uses Lyapunov's inequality to obtain the $r$th-moment estimate.  If
$F=\Qop u^{n,L,M}+\theta$, then
$F\in L^2(0,T;H^K)$ and
\begin{align}\label{eq:nLM-deterministic-convolution}
 \sup_{t\le T}\left\|\int_0^tP_{t-s}F(s)\,\dd s\right\|_{H^\alpha}^2
 &\le C\sup_{t\le T}\left(\int_0^t(t-s)^{-\alpha/2}\|F(s)\|_H\,\dd s\right)^2\notag\\
 &\le C_{\alpha,T}\int_0^T\|F(s)\|_H^2\,\dd s,
\end{align}
where $\alpha<1$ is used in the last step.  The same estimate applied to increments, together with strong continuity of $P_t$ on $H^\alpha$, gives temporal continuity.  Finally, $u_0^n$ is uniformly bounded in $H^1$, and
$H^\alpha(I)\hookrightarrow C(\overline I)$.  This proves
\eqref{eq:uniform-nLM-C} for finite and infinite $M$.

\emph{Step 4: removal of the mode cutoff.}
Let $U^M=u^{n,L,M}$, $U=u^{n,L,\infty}$, and couple them with the same cylindrical Wiener process.  Since
$G^{n,L,M}(z)=G^{n,L,\infty}(z)\Pi_M$, the exact decomposition is
\begin{align}\label{eq:mode-noise-split}
 G^{n,L,M}(U^M)-G^{n,L,\infty}(U)
 &=\bigl[G^{n,L,M}(U^M)-G^{n,L,M}(U)\bigr]
\notag\\
 &\quad-G^{n,L,\infty}(U)(I-\Pi_M).
\end{align}
Apply It\^o's formula to $\|U^M-U\|_{H^K}^2$.  The first term in
\eqref{eq:mode-noise-split} is bounded by
$C_{n,L}\|U^M-U\|_{H^K}^2$ through
\eqref{eq:Gn-global-Lip}; the second satisfies
\[
 \E\int_0^T\|G^{n,L,\infty}(U)(I-\Pi_M)\|_{L_2(\mathbb U,H^K)}^2\,\dd s
 \le C\rho_M\E\int_0^T\|U(s)\|_{H^K}^2\,\dd s.
\]
The heat term contributes
$-\sum_i\|\partial_x(U_i^M-U_i)\|_H^2$.  BDG and Young absorb one half of the supremum and gradient terms, and Gronwall gives
\[
 \E\sup_{t\le T}\|U^M(t)-U(t)\|_{H^K}^2
 +\E\int_0^T\|U^M-U\|_{V^K}^2\,\dd s
 \le C_{n,L,T}\rho_M.
\]
This is \eqref{eq:mode-truncation-convergence}.  Positivity of the full-mode solution follows either from Step~2 directly or by taking an almost-everywhere convergent subsequence of the finite-mode solutions, and its $C(\overline I)^K$-path regularity was proved directly in Step~3.
\end{proof}

The following proposition removes the amplitude cutoff and establishes global existence and uniqueness for the regularized equation.

\begin{proposition}\label{prop:remove-amplitude}
For every $n\in\N$ there is a unique global nonnegative solution $u^n$ of \eqref{eq:approx}, adapted to the completed natural filtration of the prescribed $W$, such that
\begin{equation}\label{eq:un-C-continuity}
 u^n\in L^q\bigl(\Omega;C([0,T];C(\overline I)^K)\bigr)
 \cap L^2\bigl(\Omega\times[0,T];V^K\bigr)
\end{equation}
for every finite $q$ and $T$.  The estimates \eqref{eq:uniform-nLM-energy} and \eqref{eq:uniform-nLM-C} remain valid for $u^n$, uniformly in $n$.
\end{proposition}

\begin{proof}
Set
\[
 L_0=2+\max_{1\le i\le K}\|u_{i,0}\|_{C(\overline I)}
\]
and write $u^{n,L}=u^{n,L,\infty}$ for $L\ge L_0$.  For two levels
$L_0\le L'<L$, define their common localization time
\[
 \sigma_{L'}^{n;L',L}
 =\inf\left\{t\ge0:
 \max_i\bigl(\|u_i^{n,L'}(t)\|_C\vee\|u_i^{n,L}(t)\|_C\bigr)
 \ge L'\right\}.
\]
On $[0,\sigma_{L'}^{n;L',L}]$ both state vectors take values in
$[0,L']^K$, where $g_{n,L'}=g_{n,L}$.  Their stopped difference therefore solves the same globally Lipschitz equation with the same initial state and noise.  The difference estimate used in Step~4 of
\Cref{prop:approx-existence} gives
\[
 u^{n,L'}(t)=u^{n,L}(t),
 \qquad 0\le t\le\sigma_{L'}^{n;L',L},\quad\text{a.s.}
\]
In particular, the two first exit times from the $C$-ball of radius $L'$ coincide.

For each $L$ set
\[
 \tau_L^n=
 \inf\left\{t\ge0:\max_i\|u_i^{n,L}(t)\|_{C(\overline I)}\ge L\right\}.
\]
By \eqref{eq:uniform-nLM-C}, for every $q\ge1$,
\begin{equation}\label{eq:amplitude-exit-probability}
 \sup_n\Prob(\tau_L^n\le T)
 \le L^{-q}\sup_n\E\|u^{n,L}\|_{C([0,T];C(\overline I)^K)}^q
 \le C_{q,T}L^{-q}.
\end{equation}
Take the dyadic levels $L_k=2^kL_0$ and choose $q>1$.  Then
$\sum_k\sup_n\Prob(\tau_{L_k}^n\le T)<\infty$, and Borel--Cantelli shows, for every fixed $n$ and $T$, that
$\tau_{L_k}^n>T$ for all sufficiently large $k$, almost surely.  Since
$T\in\N$ is arbitrary, $\tau_{L_k}^n\uparrow\infty$ almost surely.

The consistency just proved permits the pathwise definition
\[
 u^n(t)=u^{n,L_k}(t)\quad\text{whenever }t<\tau_{L_k}^n.
\]
It is independent of the chosen sufficiently large $k$, is nonnegative, has continuous $C(\overline I)^K$-valued paths, and solves
\eqref{eq:approx} on every bounded time interval.  Fatou's lemma applied along $k\to\infty$ gives both the variational and the $C$-path moment estimates, with the same constants as in
\eqref{eq:uniform-nLM-energy}--\eqref{eq:uniform-nLM-C}.

If $v^n$ is another nonnegative variational (equivalently, mild) solution of \eqref{eq:approx}, driven by the same $W$ and lying in
\[
 L^2\bigl(\Omega;C([0,T];H^K)\bigr)
 \cap L^2\bigl(\Omega\times[0,T];V^K\bigr),
\]
the regularized energy estimate and the factorization and deterministic-convolution argument of Step~3 first give it a continuous $C(\overline I)^K$-valued version.  We may therefore stop both processes when either $C$-norm reaches $L$.  Up to that time both solve the same globally Lipschitz $L$-truncated equation, so they agree.  Letting $L\to\infty$ proves uniqueness in this variational energy class.  Every $u^{n,L}$ is adapted to the completed natural filtration of $W$, and the patched process is therefore adapted to that filtration as well.
\end{proof}

\section{Uniform a priori estimates}\label{sec:apriori}

We establish estimates that are uniform in the boundary regularization.  The energy calculation also applies directly to every WV solution of the limiting equation, a fact used below rather than inferred from lower semicontinuity.

\begin{proposition}[Uniform variational energy estimate]\label{prop:energy-estimate}
For every $T>0$ and $q\ge1$, the regularized solutions satisfy
\begin{equation}\label{eq:uniform-energy}
 \sup_n\left[
 \E\sup_{t\le T}\|u^n(t)\|_{H^K}^{2q}
 +\E\left(\int_0^T\|u^n(t)\|_{V^K}^2\dd t\right)^q
 \right]
 \le C_{q,T}\bigl(1+\|u_0\|_{H^K}^{2q}\bigr).
\end{equation}
Every WV solution $u$ of \eqref{eq:spde} satisfies the same estimate with $u^n$ replaced by $u$.
\end{proposition}
\begin{proof}
We first take $q=1$.  Applying It\^o's formula to $\sum_i\|u_i^n(t)\|_H^2$ gives
\begin{align}\label{eq:energy-ito}
 \dd\|u^n\|_{H^K}^2
 +\sum_{i=1}^K\|\partial_xu_i^n\|_H^2\dd t
 &=2\sum_{i=1}^K\ip{u_i^n}{\sum_{j=1}^Kq_{ij}u_j^n+\theta_i}\dd t\notag\\
 &\quad+\sum_{i=1}^K\sum_{m=1}^{\infty}
 \|\sigma_{im}g_n(u_i^n)g_n(u_{i+1}^n)\|_H^2\dd t
 +\dd M_t^n,
\end{align}
where $M^n$ is a continuous local martingale.

For the matrix drift, boundedness of the coefficients  gives
\begin{align}\label{eq:matrix-energy-bound}
 2\sum_{i,j}\ip{u_i^n}{q_{ij}u_j^n}
 &\le2\sum_{i,j}\|q_{ij}\|_\infty
 \int_I|u_i^n u_j^n|\,\dd x
 \le C_{\Qmat}\|u^n\|_{H^K}^2.
\end{align}
The immigration term satisfies
\[
 2\sum_i\ip{u_i^n}{\theta_i}
 \le C_\theta\bigl(1+\|u^n\|_{H^K}^2\bigr).
\]
For the trace, $p_n(r)\le r$ yields
\begin{align}\label{eq:trace-energy-bound}
 \sum_m\|\sigma_{im}g_n(u_i^n)g_n(u_{i+1}^n)\|_H^2
 &=\int_I\lambda_i p_n(u_i^n)p_n(u_{i+1}^n)\dd x\notag\\
 &\le\|\lambda_i\|_\infty\int_Iu_i^nu_{i+1}^n\dd x\notag\\
 &\le\frac{\|\lambda_i\|_\infty}{2}
 \bigl(\|u_i^n\|_H^2+\|u_{i+1}^n\|_H^2\bigr).
\end{align}

Stop when $\|u^n\|_{H^K}$ reaches a deterministic level.  The quadratic variation of $M^n$ is controlled by a constant times
\[
 \int_0^t\|u^n(s)\|_{H^K}^2
 \sum_i\|G_i^n(u^n(s))\|_{L_2(U_i,H)}^2\dd s,
\]
which is bounded using \eqref{eq:trace-energy-bound}.  The Burkholder--Davis--Gundy and Young inequalities therefore give
\[
 \E\sup_{s\le t}|M_s^n|
 \le\frac14\E\sup_{s\le t}\|u^n(s)\|_{H^K}^2
 +C\E\int_0^t\bigl(1+\|u^n(s)\|_{H^K}^2\bigr)\dd s.
\]
Gronwall's lemma and removal of the stopping prove the estimate for $q=1$.

For the higher moments put $Y^n=1+\|u^n\|_{H^K}^2$ and
$D^n=\sum_i\|\partial_xu_i^n\|_H^2$.  The preceding estimates give, after localization,
\[
 \dd Y^n+D^n\,\dd t\le CY^n\,\dd t+\dd M^n,
 \qquad \dd\langle M^n\rangle_t\le C(Y^n)^2\,\dd t.
\]
Because $Y^n\ge1$, It\^o's formula applies to $(Y^n)^q$ for every real
$q\ge1$.  The finite-variation terms satisfy
\[
 \dd(Y^n)^q+q(Y^n)^{q-1}D^n\,\dd t
 \le C_q(Y^n)^q\,\dd t
 +q(Y^n)^{q-1}\,\dd M^n.
\]
BDG and Young yield
\begin{align}\label{eq:energy-Yq-detail}
 \E\sup_{t\le T}(Y^n(t))^q
 +\E\int_0^T(Y^n)^{q-1}D^n\,\dd t
 \le C_{q,T}(Y^n(0))^q.
\end{align}
Returning to the unpowered energy inequality,
\[
 \int_0^TD^n\,\dd t
 \le Y^n(0)+C\int_0^TY^n\,\dd t+
 \sup_{t\le T}|M^n(t)|,
\]
and hence, using
\[
 \E\sup_{t\le T}|M^n(t)|^q
 \le C_q\E\left(\int_0^T(Y^n)^2\,\dd t\right)^{q/2}
 \le C_{q,T}\E\sup_{t\le T}(Y^n(t))^q,
\]
we obtain
\[
 \E\left(\int_0^TD^n(s)\,\dd s\right)^q
 \le C_{q,T}(Y^n(0))^q.
\]
Together with \eqref{eq:energy-Yq-detail} and
$\|u^n\|_{V^K}^2\le\|u^n\|_{H^K}^2+D^n$, this proves the full estimate
\eqref{eq:uniform-energy}.  All constants have the dependence specified in
\eqref{eq:coefficient-norms} and are independent of every approximation parameter.

For an arbitrary WV solution, apply the Hilbert-space variational It\^o formula to the same stopped functionals.  The drift bounds are unchanged and
\[
 \|G_i(u)\|_{L_2(U_i,H)}^2
 =\int_I\lambda_i u_i u_{i+1}\,\dd x
 \le\frac{\|\lambda_i\|_\infty}{2}
 (\|u_i\|_H^2+\|u_{i+1}\|_H^2).
\]
The same BDG--Gronwall argument, followed by removal of the stopping time, proves the estimate directly in the weak solution class.
\end{proof}

\begin{proposition}[Total-mass moments]\label{prop:mass}
Let
\[
 \mathfrak M^n(t)=\sum_{i=1}^K\int_Iu_i^n(t,x)\dd x.
\]
Then, for every $T>0$ and $q\ge1$,
\begin{equation}\label{eq:mass-moments}
 \sup_n\E\sup_{t\le T}|\mathfrak M^n(t)|^{2q}<\infty.
\end{equation}
The same estimate holds for every WV solution.
\end{proposition}

\begin{proof}
Since $|I|=1$, we have
\[
 0\le\mathfrak M^n(t)
 \le\sqrt K\,\|u^n(t)\|_{H^K}.
\]
Thus \eqref{eq:mass-moments} follows immediately from \Cref{prop:energy-estimate}.  For a weak solution, define $\mathfrak M(t)=\sum_i\int_Iu_i(t,x)\,\dd x$ and use the identical pointwise bound.
\end{proof}

The next step is the estimates for the time increments and the compactness scale.

\begin{proposition}[Negative-Sobolev time increments]\label{prop:time-increments}
Fix $T>0$ and $q>2$.  There is $C_{q,T}$, independent of $n$, such that
\begin{equation}\label{eq:time-increments}
 \E\|u^n(t)-u^n(s)\|_{H^{-2}(I)^K}^{2q}
 \le C_{q,T}|t-s|^{q},
 \qquad 0\le s\le t\le T.
\end{equation}
Consequently, for every $\delta<(q-1)/(2q)$, the $H^{-2}$-H\"older seminorms of order $\delta$ are bounded in probability.
\end{proposition}

\begin{proof}
Since $A:V\to V^*=H^{-1}(I)$ is bounded and $H^{-1}\hookrightarrow H^{-2}$,
\[
 \E\left\|\int_s^tAu^n(r)\dd r\right\|_{H^{-2}}^{2q}
 \le |t-s|^q\E\left(\int_s^t\|u^n(r)\|_V^2\dd r\right)^q
 \le C|t-s|^q.
\]
The reaction operator is bounded on $H^K$, and the immigration is deterministic, so their increment has a $2q$th moment bounded by $C|t-s|^{2q}$.  For the stochastic increment, BDG and \eqref{eq:Gn-HS} give
\[
 \E\left\|\int_s^tG^n(u^n(r))\dd W(r)\right\|_{H^K}^{2q}
 \le C_q\E\left(\int_s^t\bigl(1+\|u^n(r)\|_{H^K}^2\bigr)\dd r\right)^q
 \le C_{q,T}|t-s|^q.
\]
This proves \eqref{eq:time-increments}, Kolmogorov's criterion gives the stated H\"older control.
\end{proof}

\begin{proposition}[Continuous spatial paths]\label{prop:spatial-continuity}
Let $z$ be either a regularized solution $u^n$ or a WV solution of \eqref{eq:spde}.  For every $T>0$, $z$ has a version in
\[
 C\bigl([0,T];C(\overline I)^K\bigr).
\]
More precisely, for every $\alpha\in(1/2,1)$ and every {$q>2/(1-\alpha)$},
\begin{equation}\label{eq:spatial-path-bound}
 \E\|z\|_{C([0,T];H^\alpha(I)^K)}^q
 \le C_{\alpha,q,T}\bigl(1+\|u_0\|_{H^1(I)^K}^q\bigr),
\end{equation}
with a constant uniform in $n$ for $z=u^n$.
\end{proposition}

\begin{proof}
Use the mild representation \eqref{eq:mild}, with $G$ replaced by $G^n$ for the regularized equation.  The initial term is continuous in $H^\alpha$ because $u_0\in H^1$; in the regularized case the same statement holds uniformly for
$u_0^n=u_0+n^{-1}\mathbf1$.

Let $F=\Qop z+\theta$.  Since $\Qop$ is bounded on $H^K$, the energy estimate gives $F\in L^2(0,T;H^K)$ with the required moments.  The analytic-semigroup estimate and Cauchy--Schwarz imply, for $\alpha<1$,
\begin{align}\label{eq:deterministic-Halpha-bound}
 \sup_{t\le T}\left\|\int_0^tP_{t-s}F(s)\,\dd s\right\|_{H^\alpha}^2
 &\le C\sup_{t\le T}
 \left(\int_0^t(t-s)^{-\alpha/2}\|F(s)\|_H\,\dd s\right)^2\notag\\
 &\le C_{\alpha,T}\int_0^T\|F(s)\|_H^2\,\dd s.
\end{align}
To see continuity, split the increment at $t-r$: on $[0,t-r]$ use strong continuity of $P_h$ on $H^\alpha$, and on $[t-r,t]$ use the same bound with a time interval of length $r$; then let first the increment and afterward $r$ tend to zero.

For the stochastic convolution, \eqref{eq:semigroup-Halpha},
\eqref{eq:G-HS} or \eqref{eq:Gn-HS}, the high energy moments of
\Cref{prop:energy-estimate}, and the factorization 
\Cref{lem:factorization} give a continuous $H^\alpha$-valued version and
\eqref{eq:spatial-path-bound}, provided $q>2/(1-\alpha)$.  The definition
\eqref{eq:Halpha-definition} and the embedding
$H^\alpha(I)\hookrightarrow C(\overline I)$ for $\alpha>1/2$ complete the proof.  Thus the entire path, including $t=0$, is continuous in
$C(\overline I)^K$.
\end{proof}

Finally, the strict-gap power entropy and boundary occupation will be estimated.
Choose $\rho\in(0,1/2)$ so small that
\begin{equation}\label{eq:rho-choice}
 \frac{\rho\Lambda}{2}\le\frac{\delta_*}{2},
\end{equation}
with the restriction omitted when $\Lambda=0$, and define
\begin{equation}\label{eq:p-choice}
 p=\frac12-\rho\in(0,1/2).
\end{equation}
Then
\begin{equation}\label{eq:cp}
 c_p:=\min_i\essinf_{x\in I}
 \left(q_{i,i+1}(x)-\frac{1-p}{2}\lambda_i(x)\right)
 \ge\frac{\delta_*}{2}>0.
\end{equation}
The term $q_{i,i+1}$ is singled out only because it has the same catalyst $u_{i+1}$ as the $i$th noise variance.  The remaining off-diagonal terms will be kept as additional nonnegative entropy production.

\begin{proposition}[Uniform power entropy]\label{prop:power-approx}
For every $T>0$, there is $C_T$ independent of $n$ and $\eta\in(0,1]$ such that, for each $i$,
\begin{align}\label{eq:power-approx}
 &\E\int_0^T\!\int_I
 (u_i^n+\eta)^{p-2}|\partial_xu_i^n|^2\dd x\dd t
 +\E\int_0^T\!\int_I
 u_{i+1}^n(u_i^n+\eta)^{p-1}\dd x\dd t\notag\\
 &\quad+\E\int_0^T\!\int_I
 \bigl(\mathcal R_i(u^n)+\theta_i\bigr)(u_i^n+\eta)^{p-1}\dd x\dd t
 \le C_T.
\end{align}
\end{proposition}

\begin{proof} We divide the proof into six steps.

\medskip
\noindent\text{Step 1: an explicit energy stopping time.}

Fix $i$ and $\eta>0$.  On $[0,\infty)$ define
\[
 f_{p,\eta}(r)=(r+\eta)^p,
 \qquad
 \Phi_{p,\eta}(w)=\int_I f_{p,\eta}(w(x))\dd x.
\]
Since $u_i^n\ge0$, extend $f_{p,\eta}$ smoothly to $\R$ so that its first two derivatives are bounded and its second derivative is uniformly continuous.  
For $R\ge1$, set
\begin{align}\label{eq:tauR-a}
 \tau_R=\inf\Biggl\{t\in[0,T]:
 &\norm{u^n(t)}_{H^K}^2
 +\int_0^t\norm{u^n(s)}_{V^K}^2\,\dd s\\
 &\quad+\sum_{\ell=1}^K\int_0^t
   \norm{G_\ell^n(u^n(s))}_{L_2(U_\ell,H)}^2\,\dd s
 \ge R\Biggr\}.
\notag
\end{align}
The regularity and energy bounds for $u^n$ imply $\tau_R\uparrow T$ almost surely.  Every term below is integrable on $[0,\tau_R]$.

Apply \Cref{lem:variational-ito} to $X=u_i^n$, with drift
\[
 Y_i^n=Au_i^n+\sum_{j=1}^Kq_{ij}u_j^n+\theta_i
\]
after separating the $V^*$-valued heat part from the $H$-valued lower-order part, and with noise coefficient $Z=G_i^n(u^n)$.  For every $t\le T$,
\begin{align}
 \Phi_{p,\eta}(u_i^n(t\wedge\tau_R))
 &=\Phi_{p,\eta}(u_{i,0}^n)
  +\int_0^{t\wedge\tau_R}
    \ip{Au_i^n}{p(u_i^n+\eta)^{p-1}}_{V^*,V}\,\dd s\notag\\
 &\quad+p\int_0^{t\wedge\tau_R}\int_I
   (u_i^n+\eta)^{p-1}
   \left(\sum_{j=1}^Kq_{ij}u_j^n+\theta_i\right)\,\dd x\,\dd s\notag\\
 &\quad-\frac{p(1-p)}2
 \int_0^{t\wedge\tau_R}\sum_{m=1}^{\infty}\int_I
 (u_i^n+\eta)^{p-2}
 \abs{G_i^n(u^n)f_m}^2\,\dd x\,\dd s\notag\\
 &\quad+M_{i,p,\eta}^n(t\wedge\tau_R),
 \label{eq:stopped-power-ito}
\end{align}
where
\begin{align}\label{eq:power-martingale}
 M_{i,p,\eta}^n(t)
 =p\sum_{m=1}^{\infty}\int_0^t
 \Biggl[\int_I
 &(u_i^n(s)+\eta)^{p-1}
 \sigma_{im}g_n(u_i^n(s))g_n(u_{i+1}^n(s))\,\dd x\Biggr]
 \dd\beta_{im}(s).
\end{align}

\medskip
\noindent\text{Step 2: the heat term and its sign.}
Using the variational definition of the Neumann Laplacian and the Sobolev chain rule,
\begin{align*}
 \ip{Au_i^n}{p(u_i^n+\eta)^{p-1}}_{V^*,V}
 &=-\frac12\int_I
 \partial_xu_i^n\,
 \partial_x\bigl[p(u_i^n+\eta)^{p-1}\bigr] \,\dd x\\
 &=-\frac{p(p-1)}2\int_I
 (u_i^n+\eta)^{p-2}\abs{\partial_xu_i^n}^2\,\dd x\\
 &=\frac{p(1-p)}2\int_I
 (u_i^n+\eta)^{p-2}\abs{\partial_xu_i^n}^2\,\dd x.
\end{align*}

\medskip
\noindent\text{Step 3: computation of the complete infinite-mode trace.}
By \eqref{eq:Gn-operator},
\begin{align*}
 \abs{G_i^n(u^n)f_m}^2
 &=\sigma_{im}^2
   g_n(u_i^n)^2g_n(u_{i+1}^n)^2=\sigma_{im}^2p_n(u_i^n)p_n(u_{i+1}^n).
\end{align*}
All factors in the last expression are nonnegative.  Therefore Tonelli's theorem and \eqref{eq:covkernel} give, pointwise in time,
\begin{align}\label{eq:trace-computation}
 \sum_{m=1}^{\infty}\int_I
 &(u_i^n+\eta)^{p-2}
 \abs{G_i^n(u^n)f_m}^2\,\dd x\notag\\
 &=\int_I
 (u_i^n+\eta)^{p-2}
 p_n(u_i^n)p_n(u_{i+1}^n)
 \sum_{m=1}^{\infty}\sigma_{im}^2\,\dd x\notag\\
 &=\int_I
 \lambda_i p_n(u_i^n)p_n(u_{i+1}^n)
 (u_i^n+\eta)^{p-2}\,\dd x.
\end{align}
The trace is finite.  Indeed, $p_n(r)\le r$ for $r\ge0$ and, because $p-2<0$,
$(u_i^n+\eta)^{p-2}\le\eta^{p-2}$.  Hence
\begin{align}\label{eq:trace-integrability}
 &\int_I\lambda_i p_n(u_i^n)p_n(u_{i+1}^n)(u_i^n+\eta)^{p-2}\,\dd x\notag\\
 &\qquad\le
 \norm{\lambda_i}_\infty\eta^{p-2}
 \int_Iu_i^nu_{i+1}^n\,\dd x
 \le\frac{\norm{\lambda_i}_\infty\eta^{p-2}}2
 \left(\norm{u_i^n}_H^2+\norm{u_{i+1}^n}_H^2\right).
\end{align}
The state energy estimate makes the time integral of the right-hand side integrable.

Substitution of the heat computation and \eqref{eq:trace-computation} into \eqref{eq:stopped-power-ito} yields the stopped pathwise identity
\begin{align}
 \Phi_{p,\eta}(u_i^n(t\wedge\tau_R))
 &=\Phi_{p,\eta}(u_{i,0}^n)
 +\frac{p(1-p)}2\int_0^{t\wedge\tau_R}\int_I
 (u_i^n+\eta)^{p-2}\abs{\partial_xu_i^n}^2\,\dd x\,\dd s\notag\\
 &\quad+p\int_0^{t\wedge\tau_R}\int_I
 (u_i^n+\eta)^{p-1}
 \left(\sum_{j=1}^Kq_{ij}u_j^n+\theta_i\right)\,\dd x\,\dd s\notag\\
 &\quad-\frac{p(1-p)}2\int_0^{t\wedge\tau_R}\int_I
 \lambda_i p_n(u_i^n)p_n(u_{i+1}^n)
 (u_i^n+\eta)^{p-2}\,\dd x\,\dd s\notag\\
 &\quad+M_{i,p,\eta}^n(t\wedge\tau_R).
 \label{eq:stopped-final}
\end{align}

\medskip
\noindent\text{Step 4: the martingale is square integrable and has mean zero.}
The adjoint form of the stochastic integrand gives
\begin{align*}
 \langle M_{i,p,\eta}^n\rangle_{t\wedge\tau_R}
 &=p^2\int_0^{t\wedge\tau_R}
 \norm{G_i^n(u^n(s))^*(u_i^n(s)+\eta)^{p-1}}_{U_i}^2\,\dd s.
\end{align*}
Since $p-1<0$ and $u_i^n\ge0$,
\[
 0<(u_i^n+\eta)^{p-1}\le\eta^{p-1}.
\]
Thus
\begin{align}\label{eq:martingale-bracket-bound}
 \langle M_{i,p,\eta}^n\rangle_{t\wedge\tau_R}
 &\le p^2\eta^{2p-2}|I|
 \int_0^{t\wedge\tau_R}
 \norm{G_i^n(u^n(s))}_{\HS(U_i,H)}^2\,\dd s
 \le p^2\eta^{2p-2}|I|R.
\end{align}
Consequently $M_{i,p,\eta}^n(\cdot\wedge\tau_R)$ is a square-integrable martingale and
\begin{equation}\label{eq:martingale-zero}
 \E M_{i,p,\eta}^n(t\wedge\tau_R)=0.
\end{equation}

Taking expectations in \eqref{eq:stopped-final} at $t=T$ therefore gives the stopped version of identity \eqref{eq:identity-6-6} below.

\medskip
\noindent\text{Step 5: removal of the stopping time.}
We verify the required uniform integrability term by term.

First, because $0<p<1$ and $\eta\le1$,
\begin{equation}\label{eq:entropy-growth}
 0\le\Phi_{p,\eta}(u_i^n(t))
 \le C\left(1+\norm{u_i^n(t)}_H^2\right).
\end{equation}
Hence the terminal entropy is uniformly integrable by the $H$-energy estimate.

Second, the weighted gradient term is bounded in absolute value by
$ \eta^{p-2}\int_0^T\norm{u_i^n(s)}_V^2\,\dd s,$
which is integrable.  Third, since $(u_i^n+\eta)^{p-1}\le\eta^{p-1}$ and the coefficients $q_{ij}$ and $\theta_i$ are bounded,
\begin{align*}
 &\int_0^T\int_I
 (u_i^n+\eta)^{p-1}
 \abs{\sum_{j=1}^Kq_{ij}u_j^n+\theta_i}\,\dd x\,\dd s\le C_{\eta,\mathfrak q_0,\mathfrak t_0,T}
 \left(1+\int_0^T\norm{u^n(s)}_{H^K}^2\,\dd s\right),
\end{align*}
where Cauchy--Schwarz and the finiteness of $|I|$ were used.  Finally, the trace is dominated by \eqref{eq:trace-integrability}.  All these dominating random variables have finite expectation.

Since $\tau_R\uparrow T$ almost surely, dominated convergence in every deterministic term of the stopped expectation identity yields

\begin{align}
 \E\Phi_{p,\eta}(u_i^n(T))
 &=\Phi_{p,\eta}(u_{i,0}^n)
 +\frac{p(1-p)}{2}\E\int_0^T\!\int_I
 (u_i^n+\eta)^{p-2}\abs{\partial_xu_i^n}^2\,\dd x\,\dd t\notag\\
 &\quad+p\E\int_0^T\!\int_I
 (u_i^n+\eta)^{p-1}
 \left(\sum_{j=1}^Kq_{ij}u_j^n+\theta_i\right)\,\dd x\,\dd t\notag\\
 &\quad-\frac{p(1-p)}{2}\E\int_0^T\!\int_I
 \lambda_i p_n(u_i^n)p_n(u_{i+1}^n)
 (u_i^n+\eta)^{p-2}\,\dd x\,\dd t.
 \label{eq:identity-6-6}
\end{align}

\medskip
\noindent\text{Step 6: finish of the proof.}

We now decompose the matrix row.  The cyclic term and the negative It\^o trace satisfy
\begin{align}\label{eq:power-combine}
 &p q_{i,i+1}u_{i+1}^n(u_i^n+\eta)^{p-1}
 -\frac{p(1-p)}{2}\lambda_i p_n(u_i^n)p_n(u_{i+1}^n)
 (u_i^n+\eta)^{p-2}\notag\\
 &\quad\ge p\left(q_{i,i+1}-\frac{1-p}{2}\lambda_i\right)
 u_{i+1}^n(u_i^n+\eta)^{p-1}\notag\\
 &\quad\ge pc_pu_{i+1}^n(u_i^n+\eta)^{p-1}.
\end{align}
Indeed, $p_n(u_i^n)\le u_i^n\le u_i^n+\eta$ and $p_n(u_{i+1}^n)\le u_{i+1}^n$; since $p-2<0$, these inequalities have exactly the direction used above.

For $j\notin\{i,i+1\}$, the Metzler condition gives
\begin{equation}\label{eq:power-other-offdiag}
 p q_{ij}u_j^n(u_i^n+\eta)^{p-1}\ge0,
\end{equation}
and we retain their sum as $p\mathcal R_i(u^n)(u_i^n+\eta)^{p-1}$.  Immigration is likewise nonnegative.  The only possible loss is the diagonal entry.  Pointwise,
\begin{align}\label{eq:power-diagonal}
 p q_{ii}u_i^n(u_i^n+\eta)^{p-1}
 &\ge-pq_{ii}^{-}u_i^n(u_i^n+\eta)^{p-1}\ge-p\|q_{ii}^{-}\|_\infty (u_i^n+\eta)^p.
\end{align}
Thus every matrix term has been accounted for: off-diagonal entries help, while the diagonal negative part has the same order as the entropy itself.

Insert \eqref{eq:power-combine}--\eqref{eq:power-diagonal} into \eqref{eq:identity-6-6} and move the three nonnegative production terms to the left.  Since $0<p<1$,
\[
 (r+\eta)^p\le C(1+r^2),\qquad r\ge0,
\]
uniformly for $0<\eta\le1$.  The energy estimate therefore bounds both the terminal entropy and the time integral of the diagonal loss.  This proves \eqref{eq:power-approx}.
\end{proof}

\begin{corollary}[Boundary occupation for all incoming parts]\label{cor:occupation-approx}
For every $0<r\le1$,
\begin{align}\label{eq:occupation-approx}
 \E\int_0^T\!\int_I
 u_{i+1}^n\one_{\{u_i^n\le r\}}\dd x\dd t
 &\le C_Tr^{1-p},\\
 \E\int_0^T\!\int_I
 \bigl(\mathcal R_i(u^n)+\theta_i\bigr)
 \one_{\{u_i^n\le r\}}\dd x\dd t
 &\le C_Tr^{1-p},\\
 \E\int_0^T\!\int_I
 |\partial_xu_i^n|^2\one_{\{u_i^n\le r\}}\dd x\dd t
 &\le C_Tr^{2-p}.
\end{align}
\end{corollary}

\begin{proof}
Apply \Cref{prop:power-approx} with $\eta=r$.  On $\{u_i^n\le r\}$,
\[
 (u_i^n+r)^{p-1}\ge(2r)^{p-1},
 \qquad
 (u_i^n+r)^{p-2}\ge(2r)^{p-2},
\]
because both exponents are negative.  Multiplication by $r^{1-p}$ or $r^{2-p}$ gives the stated estimates.
\end{proof}

\begin{proof}[Proof of \Cref{thm:power-main}]

The estimate is needed for every solution in the uniqueness class, not only for the solution produced by the approximation.  
Therefore, we need to prove it directly.  However, the proof is completely the same as that of \Cref{prop:power-approx} with $u^n$ replaced by $u$,
so we omit it.
\end{proof}

\begin{remark}[The strict exponent and the role of the matrix]\label{rem:occupation-exponent}
Since $p=1/2-\rho$,
\[
 1-p=\frac12+\rho>\frac12.
\]
This strict gain is generated only by the gap in the distinguished entry $q_{i,i+1}$.  The other off-diagonal entries do not need a positive lower bound.
  At the critical gap $\delta_*=0$, the choice $\rho>0$ is unavailable.
\end{remark}

\section{Regularized Hellinger geometry}\label{sec:hellinger-algebra}

We prove pathwise uniqueness by estimating the Hellinger distance between two synchronized WV solutions.  We begin with its regularization and the dissipative contribution of the heat operator.
For $\eta>0$, define
\begin{equation}\label{eq:heta}
 h_\eta(x,y)=\bigl(\sqrt{x+\eta}-\sqrt{y+\eta}\bigr)^2,
 \qquad x,y\ge0.
\end{equation}
Then $h_\eta\to h$ as $\eta\downarrow0$, where $h$ is defined in \eqref{eq:hellinger-distance}.  

\begin{lemma}[Convexity and heat dissipation]\label{lem:hellinger-convex}

The Hessian $D^2h_\eta(x,y)$ is positive semidefinite and has rank one.
Consequently, for every nonnegative $u,v\in V=H^1(I)$, the Sobolev
compositions $h_{\eta,x}(u,v)$ and $h_{\eta,y}(u,v)$ belong to $V$ and
\begin{align}\label{eq:hellinger-heat}
 &{}_{V^*}\!\ip{Au}{h_{\eta,x}(u,v)}_V
 +{}_{V^*}\!\ip{Av}{h_{\eta,y}(u,v)}_V\notag\\
 &\qquad=-\frac12\int_I
 \begin{pmatrix}\partial_xu\\ \partial_xv\end{pmatrix}^{\!T}
 D^2h_\eta(u,v)
 \begin{pmatrix}\partial_xu\\ \partial_xv\end{pmatrix}\dd x\le0.
\end{align}

\end{lemma}

\begin{proof}

The diagonal entries of the Hessian equal $-\frac{1}{2\sqrt{(x+\eta)(y+\eta)}}$ which is
nonnegative, and the determinant is zero, so the first statement holds.  Since $H^1(I)$ embeds into
$C(\overline I)$ in one dimension and $\eta>0$, the first derivatives of
$h_\eta$ have bounded derivatives on the compact range of $(u,v)$.  The
Sobolev chain rule therefore gives
\[
 \partial_x\bigl(h_{\eta,x}(u,v)\bigr)
 =h_{\eta,xx}(u,v)\partial_xu+h_{\eta,xy}(u,v)\partial_xv,
\]
and the analogous identity for $h_{\eta,y}(u,v)$.  Substitution into the
variational definition \eqref{eq:Avar} proves \eqref{eq:hellinger-heat}.
\end{proof}

Next, we consider the perspective identities for every entry of the matrix drift.
Fix $x,y,c,d\ge0$ and put
\begin{equation}\label{eq:UV-def}
 r=\sqrt{x+\eta},\qquad s=\sqrt{y+\eta},
 \qquad
 \alpha=\sqrt{\frac{x}{x+\eta}},\qquad
 \beta=\sqrt{\frac{y}{y+\eta}},
\end{equation}
\begin{equation}\label{eq:UV2-def}
 U=\left(\frac{y+\eta}{x+\eta}\right)^{1/4}\sqrt c,
 \qquad
 V=\left(\frac{x+\eta}{y+\eta}\right)^{1/4}\sqrt d.
\end{equation}
Here $(x,y)$ are two synchronized values in row $i$, while $(c,d)$ may be the synchronized values of any input component $j$.

\begin{lemma}[Off-diagonal perspective, diagonal growth, and immigration]\label{lem:general-drift-hellinger}
For all $x,y,c,d\ge0$,
\begin{equation}\label{eq:offdiag-perspective}
 c\partial_xh_\eta(x,y)+d\partial_yh_\eta(x,y)
 =h(c,d)-|U-V|^2\le h(c,d).
\end{equation}
Moreover,
\begin{equation}\label{eq:diagonal-hellinger}
 x\partial_xh_\eta+y\partial_yh_\eta
 =h_\eta(x,y)\left(1+\frac{\eta}{\sqrt{(x+\eta)(y+\eta)}}\right),
\end{equation}
and therefore
\begin{equation}\label{eq:diagonal-hellinger-bound}
 0\le x\partial_xh_\eta+y\partial_yh_\eta\le2h_\eta(x,y).
\end{equation}
Finally,
\begin{equation}\label{eq:immigration-contract}
 \partial_xh_\eta+\partial_yh_\eta\le0.
\end{equation}
\end{lemma}

\begin{proof}
Using $r,s$ from \eqref{eq:UV-def},
\begin{align*}
 c\partial_xh_\eta+d\partial_yh_\eta
 &=c\left(1-\frac{s}{r}\right)+d\left(1-\frac{r}{s}\right)\\
 &=c+d-2\sqrt{cd}
 -\left(\frac{s}{r}c+\frac{r}{s}d-2\sqrt{cd}\right)\\
 &=h(c,d)-|U-V|^2.
\end{align*}
For the diagonal identity, substitute $x=r^2-\eta$ and $y=s^2-\eta$:
\begin{align*}
 x\left(1-\frac{s}{r}\right)+y\left(1-\frac{r}{s}\right)
 &=(r-s)^2+\eta\left(\frac rs+\frac sr-2\right)\\
 &=(r-s)^2\left(1+\frac{\eta}{rs}\right).
\end{align*}
Since $rs\ge\eta$, \eqref{eq:diagonal-hellinger-bound} follows.  Finally,
\[
 \partial_xh_\eta+\partial_yh_\eta
 =2-\frac rs-\frac sr\le0.
\]
\end{proof}

Every off-diagonal entry has the correct sign.  Only the distinguished input $j=i+1$ must be combined with the covariance trace. Now, we consider the cyclic entry coupled to the noise trace.

\begin{lemma}[Exact regularized catalytic formula]\label{lem:regularized-catalytic}
For every $a,\lambda\ge0$,
\begin{align}\label{eq:regularized-catalytic}
 &a\bigl(c\partial_xh_\eta+d\partial_yh_\eta\bigr)
 +\frac{\lambda}{2}\bigl(
 xc\partial_{xx}h_\eta
 +2\sqrt{xycd}\,\partial_{xy}h_\eta
 +yd\partial_{yy}h_\eta
 \bigr)\notag\\
 &\qquad=a h(c,d)-a|U-V|^2+\frac{\lambda}{4}|\alpha U-\beta V|^2.
\end{align}
\end{lemma}

\begin{proof}
The first-order part is \eqref{eq:offdiag-perspective}.  By elementary calculation, we have
\begin{align*}
 &\frac12\bigl(
 xc\partial_{xx}h_\eta
 +2\sqrt{xycd}\,\partial_{xy}h_\eta
 +yd\partial_{yy}h_\eta\bigr)\\
 &\qquad=\frac14\left[
 \frac{xc\,s}{r^3}-2\frac{\sqrt{xycd}}{rs}+\frac{yd\,r}{s^3}
 \right]
 =\frac14|\alpha U-\beta V|^2.
\end{align*}
\end{proof}

The next lemma removes the cyclic boundary error.

\begin{lemma}[Boundary error under the strict gap]\label{lem:boundary-error}
Fix a state bound $R>0$ and use the exponent $p\in(0,1/2)$ chosen in \eqref{eq:p-choice}.  There exists $C_R<\infty$, depending also on the fixed coefficient bounds, such that, for every $0<\eta\le1$, $0\le x,y,c,d\le R$, $0\le\lambda\le\Lambda$, and $a-\lambda/4\ge\delta_*$,
\begin{align}\label{eq:boundary-error}
 &a\bigl(c\partial_xh_\eta+d\partial_yh_\eta\bigr)
 +\frac{\lambda}{2}\bigl(
 xc\partial_{xx}h_\eta
 +2\sqrt{xycd}\,\partial_{xy}h_\eta
 +yd\partial_{yy}h_\eta
 \bigr)\notag\\
 &\qquad\le a h(c,d)
 +C_R\eta^{1/2-p}
 \bigl[c(x+\eta)^{p-1}+d(y+\eta)^{p-1}\bigr].
\end{align}
\end{lemma}

\begin{proof}
Set $L=\lambda/4$.

Choose $\zeta>0$, depending only on $\delta_*$ and $\Lambda$, as follows.
If $\Lambda=0$, take $\zeta=1$ (then $L=0$); if $\Lambda>0$, take
$\zeta=\min\{1,2\delta_*/\Lambda\}$.  Since $L\le\Lambda/4$ and
$a-L\ge\delta_*$, this choice gives $L\zeta\le\delta_*/2$ and hence

\[
 L(1+\zeta)\le a-\frac{\delta_*}{2}.
\]
Write
\[
 \alpha U-\beta V=(U-V)-\bigl[(1-\alpha)U-(1-\beta)V\bigr].
\]
Using $|r-s|^2\le(1+\zeta)|r|^2+(1+\zeta^{-1})|s|^2$ in \eqref{eq:regularized-catalytic} gives
\begin{align}\label{eq:error-young}
 L|\alpha U-\beta V|^2-a|U-V|^2
 &\le-\frac{\delta_*}{2}|U-V|^2\notag\\
 &\quad+C\bigl[(1-\alpha)^2U^2+(1-\beta)^2V^2\bigr].
\end{align}
The negative term may be discarded.  Since
\[
 1-\alpha\le\sqrt{\frac{\eta}{x+\eta}},
 \qquad U^2=c\sqrt{\frac{y+\eta}{x+\eta}},
\]
state localization yields
\[
 (1-\alpha)^2U^2\le C_R\frac{\eta c}{(x+\eta)^{3/2}}.
\]
Because $p<1/2$ and $x+\eta\ge\eta$,
\[
 \frac{\eta}{(x+\eta)^{3/2}}
 \le\eta^{1/2-p}(x+\eta)^{p-1}.
\]
The same estimate holds for the $\beta V$ term.  Substitution into \eqref{eq:error-young} proves \eqref{eq:boundary-error}.
\end{proof}

Let $u$ and $v$ be synchronized nonnegative WV solutions.  By \Cref{prop:spatial-continuity}, both have continuous $C(\overline I)^K$-valued paths.

For the deterministic initial states in \Cref{thm:stability}, put
\[
 R_0=\max_{1\le i\le K}\bigl(
 \|u_i(0)\|_{C(\overline I)}+\|v_i(0)\|_{C(\overline I)}\bigr),
\]
and, for $R>R_0$, define

\begin{equation}\label{eq:tauR}
 \tau_R=\inf\left\{t\ge0:
 \max_i\bigl(\|u_i(t)\|_{C(\overline I)}+\|v_i(t)\|_{C(\overline I)}\bigr)\ge R
 \right\}\wedge T,
\end{equation}
and define
\begin{equation}\label{eq:Heta}
 \cH_\eta(t)=\sum_{i=1}^K\int_Ih_\eta(u_i(t,x),v_i(t,x))\dd x.
\end{equation}

\begin{proposition}[Stopped Hellinger inequality]\label{prop:stopped-hellinger}

For every $R>R_0$ and $t\le T$,

\begin{align}\label{eq:stopped-hellinger}
 \E\cH_\eta(t\wedge\tau_R)
 &\le\E\cH_\eta(0)+\E\int_0^{t\wedge\tau_R}\!\int_I
 \sum_{i=1}^K2q_{ii}^{+}h_\eta(u_i,v_i)\dd x\dd s\notag\\
 &\quad+\E\int_0^{t\wedge\tau_R}\!\int_I
 \sum_{i=1}^K\sum_{j\ne i}q_{ij}h(u_j,v_j)\dd x\dd s\\
 &\quad+C_R\eta^{1/2-p}\E\int_0^{t\wedge\tau_R}\!\int_I
 \sum_{i=1}^K\bigl[u_{i+1}(u_i+\eta)^{p-1}
 +v_{i+1}(v_i+\eta)^{p-1}\bigr]\dd x\dd s.\notag
\end{align}
\end{proposition}

\begin{proof}
Fix $R>R_0$ and $\eta$.
Choose a smooth map
$\kappa_{R,\eta}:\R\to[-\eta/2,R+1]$ that equals the identity on an open neighborhood of $[0,R]$, and define
\[
 \widehat h_{\eta,R}(x,y)
 =\bigl(\sqrt{\kappa_{R,\eta}(x)+\eta}
       -\sqrt{\kappa_{R,\eta}(y)+\eta}\bigr)^2.
\]
Then $\widehat h_{\eta,R}\in C^2(\R^2)$ has bounded gradient and Hessian, its Hessian is uniformly continuous, and it agrees together with all first and second derivatives with $h_\eta$ on $[0,R]^2$.  Since every stopped value $(u_i,v_i)$ lies in that rectangle, the two-component variational It\^o formula in \Cref{lem:variational-ito} may be applied with $\widehat h_{\eta,R}$, after which all derivatives can be written as derivatives of $h_\eta$.

Summing the rowwise formulas gives the following full stopped identity, before any sign is discarded:
\begin{align}\label{eq:full-stopped-hellinger-identity}
 &\cH_\eta(t\wedge\tau_R)
 =\cH_\eta(0)
 -\frac12\sum_{i=1}^K\int_0^{t\wedge\tau_R}\!\int_I
 \begin{pmatrix}\partial_xu_i\\ \partial_xv_i\end{pmatrix}^{\!T}
 D^2h_\eta(u_i,v_i)
 \begin{pmatrix}\partial_xu_i\\ \partial_xv_i\end{pmatrix}
 \,\dd x\dd s\notag\\
 &\quad+\sum_{i=1}^K\int_0^{t\wedge\tau_R}\!\int_I
 \Biggl[
 h_{\eta,x}(u_i,v_i)
 \left(\sum_{j=1}^Kq_{ij}u_j+\theta_i\right)
 +h_{\eta,y}(u_i,v_i)
 \left(\sum_{j=1}^Kq_{ij}v_j+\theta_i\right)
 \Biggr]\,\dd x\dd s\notag\\
 &\quad+\frac12\sum_{i=1}^K\int_0^{t\wedge\tau_R}\!\int_I\lambda_i
 \Bigl[
 u_i u_{i+1}h_{\eta,xx}(u_i,v_i)
 +2\sqrt{u_i u_{i+1}v_i v_{i+1}}\,h_{\eta,xy}(u_i,v_i)\notag\\
 &\hspace{13.6em}
 +v_i v_{i+1}h_{\eta,yy}(u_i,v_i)
 \Bigr] \,\dd x\dd s
 +\mathcal M_{\eta,R}(t),
\end{align}
where
\begin{align}\label{eq:hellinger-martingale}
 \mathcal M_{\eta,R}(t)
 =\sum_{i=1}^K\sum_{m=1}^{\infty}
 \int_0^{t\wedge\tau_R}
 \Biggl\{\int_I\sigma_{im}\Bigl[
 &h_{\eta,x}(u_i,v_i)\sqrt{u_i u_{i+1}}\notag\\
 &+h_{\eta,y}(u_i,v_i)\sqrt{v_i v_{i+1}}
 \Bigr]\,\dd x\Biggr\}\,\dd\beta_{im}(s).
\end{align}
The mixed term in the third line of
\eqref{eq:full-stopped-hellinger-identity} comes from the common Brownian coordinate $\beta_{im}$ in the two synchronized equations.  The complete mode sum is legitimate because the product formula gives the absolutely convergent estimate
\begin{align*}
 &\sum_m\int_I
 \left|h_{\eta,xy}(u_i,v_i)\sigma_{im}^2
 \sqrt{u_i u_{i+1}v_i v_{i+1}}\right|\,\dd x\\
 &\qquad\le \|h_{\eta,xy}\|_{L^\infty([0,R]^2)}
 \|\lambda_i\|_\infty
 \int_I\sqrt{u_i u_{i+1}v_i v_{i+1}}\,\dd x<\infty
\end{align*}
under the stopping.  The two diagonal trace sums are handled identically.

The stopped martingale is a true square-integrable martingale.  Indeed, with
$C_{\eta,R}=\|\nabla\widehat h_{\eta,R}\|_\infty$, its bracket satisfies
\begin{align}\label{eq:hellinger-bracket}
 \E\langle\mathcal M_{\eta,R}\rangle_T
 &\le2C_{\eta,R}^2|I|
 \sum_{i=1}^K\E\int_0^{T\wedge\tau_R}
 \Bigl(\|G_i(u)\|_{L_2(U_i,H)}^2
       +\|G_i(v)\|_{L_2(U_i,H)}^2\Bigr)\,\dd s\notag\\
 &\le C_{\eta,R,T}\left(
 1+\E\int_0^T\|u(s)\|_{H^K}^2+\|v(s)\|_{H^K}^2\,\dd s\right)<\infty,
\end{align}
where \eqref{eq:G-HS} and \Cref{prop:energy-estimate} were used.  Hence
$\E\mathcal M_{\eta,R}(t)=0$.

We now estimate the deterministic terms in
\eqref{eq:full-stopped-hellinger-identity}.  The heat term is nonpositive by
\Cref{lem:hellinger-convex}.  The diagonal drift pair satisfies
\[
 q_{ii}\bigl(u_i h_{\eta,x}+v_i h_{\eta,y}\bigr)
 \le2q_{ii}^{+}h_\eta(u_i,v_i)
\]
by \eqref{eq:diagonal-hellinger-bound}.  For $j\ne i,i+1$, the Metzler condition and \eqref{eq:offdiag-perspective} give
\[
 q_{ij}\bigl(u_jh_{\eta,x}+v_jh_{\eta,y}\bigr)
 \le q_{ij}h(u_j,v_j).
\]
The cyclic pair $q_{i,i+1}(u_{i+1},v_{i+1})$, together with the complete trace in its row, is bounded by \Cref{lem:boundary-error}.  Immigration is nonpositive by \eqref{eq:immigration-contract}.  Taking expectations in the full identity, using the zero mean of the martingale, and then summing these estimates proves
\eqref{eq:stopped-hellinger}.  No spatial Galerkin approximation or positivity-preserving projection is used.
\end{proof}

\begin{proof}[Proof of \Cref{thm:stability}]
The power entropy bounds the last line of \eqref{eq:stopped-hellinger} by $C_{R,T}\eta^{1/2-p}$, which tends to zero because $p<1/2$.  Moreover,
\begin{equation}\label{eq:heta-domination}
 0\le h_\eta(x,y)\le x+y,
 \qquad h_\eta(x,y)\longrightarrow h(x,y).
\end{equation}

Because $R>R_0$ and the stopped $C(\overline I)^K$-valued paths are
continuous, at every $s\le T$,
\[
 0\le \cH_\eta(s\wedge\tau_R)
 \le \sum_{i=1}^K\int_I(u_i+v_i)(s\wedge\tau_R)\,\dd x
 \le K|I|R.
\]
Likewise, on $\{s\le\tau_R\}$ every $h_\eta(u_i,v_i)$ and
$h(u_i,v_i)$ is bounded by $u_i+v_i\le R$ pointwise.  After multiplication
by the bounded coefficients and integration over $[0,T]\times I$, this
gives deterministic integrable dominators for all Hellinger terms in the
stopped inequality.  Thus the family used in the passage
$\eta\downarrow0$ is uniformly integrable; the power-entropy estimate
handles the remaining boundary-error term.

Dominated convergence therefore gives
\begin{align}\label{eq:pre-gronwall-matrix}
 \E\cH(u(t\wedge\tau_R),v(t\wedge\tau_R))
 &\le\E\cH(u(0),v(0))\notag\\
 &\quad+\E\int_0^t\one_{\{s\le\tau_R\}}\int_I
 \sum_{j=1}^K
 \left(2q_{jj}^{+}+\sum_{i\ne j}q_{ij}\right)
 h(u_j,v_j)\dd x\dd s.
\end{align}
The reindexing is exact:
\[
 \sum_i\sum_{j\ne i}q_{ij}h(u_j,v_j)
 =\sum_j\left(\sum_{i\ne j}q_{ij}\right)h(u_j,v_j).
\]
For every $s$ one has the elementary pointwise inequality
\begin{equation}\label{eq:stopped-indicator-domination}
 \one_{\{s\le\tau_R\}}\cH(u(s),v(s))
 \le\cH(u(s\wedge\tau_R),v(s\wedge\tau_R)).
\end{equation}
By the definition of
$\mathfrak q_{\mathrm H}$, \eqref{eq:pre-gronwall-matrix} and
\eqref{eq:stopped-indicator-domination} therefore give
\[
 \E\cH(u(t\wedge\tau_R),v(t\wedge\tau_R))
 \le\E\cH(u(0),v(0))
 +\mathfrak q_{\mathrm H}\int_0^t
 \E\cH(u(s\wedge\tau_R),v(s\wedge\tau_R))\dd s.
\]
Gronwall's lemma yields
\[
 \E\cH(u(t\wedge\tau_R),v(t\wedge\tau_R))
 \le e^{\mathfrak q_{\mathrm H}t}\E\cH(u(0),v(0)).
\]

The stopping times increase to $T$ almost surely.  More explicitly, for
every $R>R_0$,
\[
 \cH\bigl(u(t\wedge\tau_R),v(t\wedge\tau_R)\bigr)
 \le |I|^{1/2}\sum_{i=1}^K\left(
   \sup_{s\le T}\|u_i(s)\|_H+\sup_{s\le T}\|v_i(s)\|_H\right).
\]
The random variable on the right is integrable by the energy estimate and
is independent of $R$.  The same total-mass dominator, integrated in time,
controls the Gronwall integrands.  Dominated convergence as $R\to\infty$
therefore proves \eqref{eq:stability-main}.

If the initial states agree, then the expected Hellinger distance is zero at every fixed time.  On one event of probability one, this holds at every rational time.  Since $h(r,s)=0$ if and only if $r=s$ and both paths are continuous in $H^K$, equality extends from the rational times to all $t\in[0,T]$.  Thus the two processes are indistinguishable.
\end{proof}

\section{Compactness and identification of weak limits}\label{sec:compactness}

There are two distinct limiting procedures.  The mode cutoff has already been removed for each fixed boundary regularization in \Cref{sec:approx}.  In this section only $n\to\infty$ is considered.  The state and the complete countable noise are kept together so that two coordinates of a joint subsequential limit are driven by the same limiting Wiener process.

Fix $T>0$ and set
\begin{equation}\label{eq:path-space}
 \cX_T=L^2(0,T;H^K)\cap C([0,T];H_A^{-2}(I)^K),
 \qquad
 \cW_T=C([0,T];\mathbb U_0),
\end{equation}
with their sum norms.  To verify the assertion used below, identify
$\cX_T$ with the set of pairs $(x,y)$ in
\[
 L^2(0,T;H^K)\times C([0,T];H_A^{-2}(I)^K)
\]
whose images agree in $L^2(0,T;H_A^{-2}(I)^K)$.  This is the kernel of a
continuous linear map and hence a closed subspace of a product of Polish
spaces.  Thus $\cX_T$ is Polish; $\cW_T$ is Polish as a continuous-path
space over the separable Hilbert space $\mathbb U_0$.  The duality pairing
of a path in $C([0,T];H_A^{-2})$ with a test function in $D(A)=H_A^2$ is
continuous.  This is the pairing used throughout the compactness and
martingale-identification arguments below.

\begin{proposition}[Tightness of the state--noise laws]\label{prop:tightness}
The  joint laws of $(u^n,W)$ are tight on $\cX_T\times\cW_T$.
\end{proposition}

\begin{proof}
The uniform bounds in
$L^2(0,T;V^K)$ and $L^\infty(0,T;H^K)$ follow from \Cref{prop:energy-estimate}.  The increment estimate \eqref{eq:time-increments} and Kolmogorov's criterion give, with arbitrarily high probability, a common H\"older modulus in $H_A^{-2}$.  The deterministic compactness statement \Cref{lem:compactness-appendix}, based on $V^K\Subset H^K\Subset H_A^{-2}(I)^K$ and Simon's compactness theorem \cite{Simon1987}, then gives compact subsets of $\cX_T$ carrying arbitrarily large probability.

The partial sums of $W$ are Cauchy in
$L^2(\Omega;C([0,T];\mathbb U_0))$, because
\[
 \E\sup_{t\le T}\left\|\sum_{m=N}^M\beta_{im}(t)f_m\right\|_{U_{i,0}}^2
 \le4T\sum_{m=N}^M(1+m)^{-2s_0}.
\]
Thus the law of $W$ is a fixed Radon probability measure on $\cW_T$.  Tightness of the joint laws follows from tightness of the two marginals.
\end{proof}

For $z\ge0$, retain the notation $G_i(z)$ from \eqref{eq:G-operator} and $G_i^n(z)$ from \eqref{eq:Gn-operator}.

\begin{lemma}[Full Hilbert--Schmidt coefficient convergence]\label{lem:coefficient-convergence}
Let $z^n,z\ge0$, suppose
$z^n\to z$ in $L^2(0,T;H^K)$.  Then
\begin{equation}\label{eq:B-convergence}
 G_i^n(z^n)\longrightarrow G_i(z)
 \quad\text{in }L^2\bigl(0,T;L_2(U_i,H)\bigr)
\end{equation}
for every $i$.  If the state convergence holds in probability and the relevant $L^2$ norms have a uniformly bounded moment of order greater than two, then \eqref{eq:B-convergence} holds in probability and its squared errors are uniformly integrable.
\end{lemma}

\begin{proof}
By \eqref{eq:product-approx},
\begin{align*}
 &\int_0^T\|G_i^n(z^n)-G_i(z^n)\|_{L_2(U_i,H)}^2\,\dd t\\
 &\qquad\le2\eps_n\|\lambda_i\|_\infty
 \int_0^T\!\int_I(z_i^n+z_{i+1}^n)\,\dd x\dd t\longrightarrow0.
\end{align*}
For $a,b,c,d\ge0$,
\[
 |\sqrt{ab}-\sqrt{cd}|^2\le2b|a-c|+2c|b-d|.
\]
Therefore,
\begin{align*}
 &\int_0^T\|G_i(z^n)-G_i(z)\|_{L_2(U_i,H)}^2\,\dd t\\
 &\quad\le2\|\lambda_i\|_\infty\left[
 \|z_{i+1}^n\|_{L^2(0,T;H)}\|z_i^n-z_i\|_{L^2(0,T;H)}
 +\|z_i\|_{L^2(0,T;H)}\|z_{i+1}^n-z_{i+1}\|_{L^2(0,T;H)}
 \right],
\end{align*}
which tends to zero.  The probabilistic assertion follows from these bounds and Vitali's theorem.
\end{proof}

\begin{proposition}[Identification of every subsequential limit]\label{prop:limit-identification}
Suppose a subsequence of $(u^n,W)$ converges in law on $\cX_T\times\cW_T$ to $(u,W)$.  Relative to the completed, right-continuous canonical filtration generated jointly by $(u,W)$:
\begin{enumerate}[label=\textup{(\roman*)},leftmargin=2.2em]
\item the coordinates of $W$ are mutually independent Brownian motions, and $W_i$ is a cylindrical Wiener process on $U_i$;
\item $u$ is a probabilistically weak WV solution of \eqref{eq:spde} driven by this same $W$;
\item $u$ satisfies \eqref{eq:main-energy} and \eqref{eq:power-main};
\item $u$ has continuous paths in both $H^K$ and $C(\overline I)^K$.
\end{enumerate}
\end{proposition}

\begin{proof}
By the Skorokhod representation theorem, pass to a new probability space on which random variables with the same laws satisfy
\begin{equation}\label{eq:skorokhod-convergence}
 (\widetilde u^n,\widetilde W^n)\longrightarrow(\widetilde u,\widetilde W)
 \quad\text{a.s. in }\cX_T\times\cW_T.
\end{equation}
We omit tildes below.

\emph{Step 1: variational regularity on the new space.}
The equality of laws and \eqref{eq:uniform-energy} imply boundedness of $(u^n)$ in
$L^2(\Omega\times(0,T);V^K)$.  Extract a weakly convergent subsequence in that Hilbert space.  On the other hand, \eqref{eq:skorokhod-convergence} gives convergence in probability in $L^2(0,T;H^K)$, and the higher energy moments imply uniform integrability; hence the convergence is strong in
$L^2(\Omega\times(0,T);H^K)$.  The weak $V^K$-limit is therefore $u$, and
\[
 u\in L^2(\Omega\times(0,T);V^K).
\]
Nonnegativity passes through the almost-everywhere $L^2$ limit.

Moreover, evaluation at zero is continuous from $\cX_T$ to
$H_A^{-2}(I)^K$.  Since $u^n(0)=u_0^n$ and $u_0^n\to u_0$ in
$H^K$, \eqref{eq:skorokhod-convergence} also gives
\begin{equation}\label{eq:limit-initial-condition}
 u(0)=u_0\qquad\text{in }H_A^{-2}(I)^K\quad\text{a.s.}.
\end{equation}

\emph{Step 2: Brownian coordinates in the joint filtration.}

Let
\[
 \mathcal G_t^0=\sigma\{u(r),W(r):0\le r\le t\},
\]
and
put $\overline{\mathcal G}_t^0=\mathcal G_t^0\vee\mathcal N$, where $\mathcal N$ denotes the null sets of the terminal canonical law.  We use the
usual augmentation
\[
 \mathcal G_t=
 \bigcap_{\substack{r\in\mathbb Q,\ t<r\le T}}
 \overline{\mathcal G}_r^0\quad(0\le t<T),
 \qquad \mathcal G_T=\overline{\mathcal G}_T^0.
\]
Fix one countable graph-norm dense set $\mathcal D_0\subset D(A)$.
Let $\mathscr C_s$ be the countable algebra of bounded cylinder random
variables of the form
\begin{equation}\label{eq:canonical-cylinder-class}
 F=\Phi\Bigl(
   \ip{u(r_1)}{\phi_1},\ldots,\ip{u(r_a)}{\phi_a},
   \beta_{i_1m_1}(r'_1),\ldots,\beta_{i_bm_b}(r'_b)
 \Bigr),
\end{equation}
where all $r_\ell,r'_\ell\in\mathbb Q\cap[0,s]$, all
$\phi_\ell\in\mathcal D_0^K$, the state brackets are the
$H_A^{-2}$--$D(A)$ dualities, and $\Phi$ ranges over a fixed countable
convergence-determining algebra of bounded Lipschitz functions on
finite-dimensional Euclidean spaces.

For completeness, this single test set determines the entire state-path
$\sigma$-field.  Indeed, for every $a\le T$ the canonical injection
\[
 L^2(0,a;H^K)\cap C([0,a];H_A^{-2}(I)^K)
 \longrightarrow C([0,a];H_A^{-2}(I)^K)
\]
is continuous and injective between Polish spaces.  The Lusin--Souslin
theorem therefore makes its inverse on its image Borel measurable.
Consequently the Borel $\sigma$-field of the intersection space is generated
by its continuous-path coordinate.  Graph-norm density of $\mathcal D_0$
implies that its duality pairings separate points of $H_A^{-2}$, and
rational-time evaluations generate the Borel $\sigma$-field of the continuous
path space.  Together with the analogous rational coordinate evaluations
of $W$, this proves that $\mathscr C_s$ generates $\mathcal G_s^0$.

For rational $0\le s<t\le T$, $F\in\mathscr C_s$, and any two noise coordinates, the approximating joint law satisfies
\begin{align}
 \E\bigl[(\beta_{im}^n(t)-\beta_{im}^n(s))F^n\bigr]&=0,
 \label{eq:brownian-increment-identity-n}\\
 \E\Bigl[\bigl(
  \beta_{im}^n(t)\beta_{jq}^n(t)-\delta_{ij}\delta_{mq}t
  -\beta_{im}^n(s)\beta_{jq}^n(s)+\delta_{ij}\delta_{mq}s
 \bigr)F^n\Bigr]&=0,
 \label{eq:brownian-product-identity-n}
\end{align}
where $F^n$ is obtained from \eqref{eq:canonical-cylinder-class} by replacing $(u,W)$ with $(u^n,W^n)$.  These equalities depend only on the joint law: on the original space $u^n$ is adapted to $W^n$, and hence they remain valid for the Skorokhod copies even though adaptedness is not asserted on the representation space.

The almost-sure convergence \eqref{eq:skorokhod-convergence} implies
$F^n\to F$ and convergence of every displayed noise coordinate.  Since
$F^n$ is uniformly bounded and Brownian suprema have moments of every
order, both families in
\eqref{eq:brownian-increment-identity-n}--\eqref{eq:brownian-product-identity-n}
are uniformly integrable.  Passing to the limit therefore gives the two
identities for rational $s,t$ and every $F\in\mathscr C_s$.  The functional
monotone-class theorem extends them to every bounded
$\mathcal G_s^0$-measurable random variable, and path continuity first
extends the identities to all deterministic $s<t$ for the completed raw
filtration $(\overline{\mathcal G}_t^0)$.

We spell out the final augmentation step.  If $Y$ is bounded and
$\mathcal G_s$-measurable, then it is
$\overline{\mathcal G}_{s_k}^0$-measurable for every rational sequence
$s_k\downarrow s$.  Applying the raw-filtration identity at $s_k<t$ and
then using path continuity and uniform integrability gives the identity at
$s$.  Completion does not change these conditional-expectation
identities.  Thus every finite family of coordinates is a continuous
$(\mathcal G_t)$-local martingale whose covariance matrix is the identity.
The multidimensional L\'evy characterization, see
\cite[Chapter~IV]{RevuzYor1999}, shows that the coordinates are mutually
independent $(\mathcal G_t)$-Brownian motions.  Consequently each $W_i$ is
a cylindrical Wiener process on $U_i$ in the enlarged joint filtration.
This proves (i).

\emph{Step 3: martingales, brackets, cross-variations, and reconstruction of the full stochastic integral.}
For $\varphi\in\mathcal D_0$, define
\begin{align}\label{eq:Mnphi}
 M_i^{n,\varphi}(t)
 &=\ip{u_i^n(t)}{\varphi}-\ip{u_{i,0}^n}{\varphi}
 -\int_0^t\ip{u_i^n(s)}{A\varphi}\,\dd s\notag\\
 &\quad-\int_0^t\ip{\sum_{j=1}^Kq_{ij}u_j^n(s)+\theta_i}{\varphi}\,\dd s.
\end{align}

Define $M_i^\varphi$ by the same formula with $(u^n,u_0^n)$ replaced by
$(u,u_0)$.  The map from $\cX_T$ to the displayed continuous process is
continuous: the endpoint terms use the $C(H_A^{-2})$ topology and
$\varphi\in D(A)$, while the time integrals use the strong $L^2(H)$
topology; the deterministic initial terms converge because
$u_0^n\to u_0$ in $H^K$.  Hence
\begin{equation}\label{eq:M-process-convergence}
 M_i^{n,\varphi}\longrightarrow M_i^\varphi
 \quad\text{a.s. in }C([0,T]).
\end{equation}
The continuous adapted $H_A^{-2}$-valued path $u$ is progressively
measurable.  The embedding $H^K\hookrightarrow H_A^{-2}(I)^K$ is
continuous and injective, so Lusin--Souslin provides a Borel inverse on
its image.  Together with $u\in L^2(\widetilde\Omega\times(0,T);H^K)$,
this gives an $H^K$-valued predictable representative after modification
on a $\dd\widetilde{\Prob}\,\dd t$-null set.  Since $z\mapsto G_i(z)$ is Borel
(and continuous on the nonnegative cone with the $H^K$ topology),
$G_i(u)$ is a predictable $L_2(U_i,H)$-valued process.

For $\varphi,\psi\in\mathcal D_0$, put
\begin{align}
 A_{ij}^{n,\varphi,\psi}(t)
 &=\delta_{ij}\int_0^t
 \ip{G_i^n(u^n(s))^*\varphi}{G_i^n(u^n(s))^*\psi}_{U_i}\,\dd s,
 \label{eq:brackets-n}\\
 C_{ijq}^{n,\varphi}(t)
 &=\delta_{ij}\int_0^t\ip{G_i^n(u^n(s))f_q}{\varphi}\,\dd s.
 \label{eq:cross-n}
\end{align}
Let $A_{ij}^{\varphi,\psi}$ and $C_{ijq}^{\varphi}$ denote the analogous expressions with $G_i(u)$.  By \Cref{lem:coefficient-convergence},
\begin{equation}\label{eq:bracket-integrator-convergence}
 A_{ij}^{n,\varphi,\psi}\to A_{ij}^{\varphi,\psi},
 \qquad
 C_{ijq}^{n,\varphi}\to C_{ijq}^{\varphi}
 \quad\text{in probability in }C([0,T]),
\end{equation}
and their total variations are uniformly integrable.  Here are the moment
bounds used in this assertion and in the passage through the martingale
identities below.  By \eqref{eq:G-HS} and
\eqref{eq:uniform-energy}, for each fixed collection of indices and test
functions there is $\delta>0$ such that
\begin{align}\label{eq:UI-identification}
 \sup_n\E\sup_{t\le T}|M_i^{n,\varphi}(t)|^{2+4\delta}&<\infty,
 \notag\\
 \sup_n\E\Bigl[\operatorname{Var}_{[0,T]}
       (A_{ij}^{n,\varphi,\psi})\Bigr]^{1+2\delta}
 +\sup_n\E\Bigl[\operatorname{Var}_{[0,T]}
       (C_{ijq}^{n,\varphi})\Bigr]^{1+2\delta}&<\infty.
\end{align}
Indeed,
\begin{align*}
 \langle M_i^{n,\varphi}\rangle_T
 &\le \|\varphi\|_H^2\int_0^T
       \|G_i^n(u^n(s))\|_{L_2(U_i,H)}^2\,\dd s,\\
 \operatorname{Var}_{[0,T]}(A_{ij}^{n,\varphi,\psi})
 &\le \frac12\int_0^T\bigl(
       \|G_i^n(u^n)^*\varphi\|_{U_i}^2
       +\|G_i^n(u^n)^*\psi\|_{U_i}^2\bigr)\,\dd s,\\
 \operatorname{Var}_{[0,T]}(C_{ijq}^{n,\varphi})
 &\le T^{1/2}\|\varphi\|_H
       \left(\int_0^T
       \|G_i^n(u^n(s))\|_{L_2(U_i,H)}^2\,\dd s\right)^{1/2},
\end{align*}
and BDG completes the first estimate.  Since Brownian suprema have moments
of every order, H\"older's inequality also gives
\begin{equation}\label{eq:UI-products-identification}
 \sup_n\E\sup_{t\le T}
 \left(
 |M_i^{n,\varphi}(t)M_j^{n,\psi}(t)|^{1+\delta}
 +|M_i^{n,\varphi}(t)\beta_{jq}^n(t)|^{1+\delta}
 \right)<\infty.
\end{equation}
Thus the brackets, cross-variation compensators, and every product
multiplier below are uniformly integrable.

For rational $0\le s<t\le T$ and $F^n$ as above, the approximating joint law satisfies the three identities
\begin{align}
 \E\bigl[(M_i^{n,\varphi}(t)-M_i^{n,\varphi}(s))F^n\bigr]&=0,
 \label{eq:martingale-test-n}\\
 \E\Bigl[\bigl(
 M_i^{n,\varphi}(t)M_j^{n,\psi}(t)-A_{ij}^{n,\varphi,\psi}(t)
 -M_i^{n,\varphi}(s)M_j^{n,\psi}(s)+A_{ij}^{n,\varphi,\psi}(s)
 \bigr)F^n\Bigr]&=0,
 \label{eq:bracket-test-n}\\
 \E\Bigl[\bigl(
 M_i^{n,\varphi}(t)\beta_{jq}^n(t)-C_{ijq}^{n,\varphi}(t)
 -M_i^{n,\varphi}(s)\beta_{jq}^n(s)+C_{ijq}^{n,\varphi}(s)
 \bigr)F^n\Bigr]&=0.
 \label{eq:cross-test-n}
\end{align}

The explicit bounds
\eqref{eq:UI-identification}--\eqref{eq:UI-products-identification} permit
passage to the limit in all three identities using
\eqref{eq:M-process-convergence}--\eqref{eq:bracket-integrator-convergence}.
The monotone-class argument first gives the identities in the completed
raw filtration.  The rational right-limit argument in Step~2 then extends
them to $(\mathcal G_t)$.  Simultaneously for all
$\varphi,\psi\in\mathcal D_0$, all modes, and all times,
$M_i^\varphi$ is therefore a continuous square-integrable
$(\mathcal G_t)$-martingale and

\begin{equation}\label{eq:limit-bracket-cross}
 \langle M_i^\varphi,M_j^\psi\rangle=A_{ij}^{\varphi,\psi},
 \qquad
 \langle M_i^\varphi,\beta_{jq}\rangle=C_{ijq}^{\varphi}.
\end{equation}
The simultaneity is obtained by first intersecting the probability-one events over the countable set of rational times, test functions, components, and modes, and then using continuity.

For $N\ge1$, define the finite-mode integral
\[
 I_i^{\varphi,N}(t)
 =\sum_{q=1}^N\int_0^t
   \ip{G_i(u(s))f_q}{\varphi}\,\dd\beta_{iq}(s).
\]
The stochastic-integration rule for covariations and \eqref{eq:limit-bracket-cross} give
\begin{align}
 \langle M_i^\varphi,I_i^{\varphi,N}\rangle_t
 &=\int_0^t\sum_{q=1}^N
   \abs{\ip{G_i(u(s))f_q}{\varphi}}^2\,\dd s,
 \label{eq:finite-cross-reconstruction}\\
 \langle I_i^{\varphi,N}\rangle_t
 &=\int_0^t\sum_{q=1}^N
   \abs{\ip{G_i(u(s))f_q}{\varphi}}^2\,\dd s.
 \label{eq:finite-bracket-reconstruction}
\end{align}
Moreover, the BDG inequality yields
\begin{align}\label{eq:full-integral-L2-limit}
 &\E\sup_{t\le T}\abs{I_i^{\varphi,M}(t)-I_i^{\varphi,N}(t)}^2\notag\\
 &\qquad\le4\E\int_0^T\sum_{q=N+1}^M
 \abs{\ip{G_i(u(s))f_q}{\varphi}}^2\,\dd s
 \longrightarrow0.
\end{align}
The limit is precisely
\[
 I_i^\varphi(t)=\int_0^t\ip{G_i(u(s))^*\varphi}{\dd W_i(s)}_{U_i}.
\]
The right-hand sides of \eqref{eq:finite-cross-reconstruction}--\eqref{eq:finite-bracket-reconstruction} converge in $L^1$ and in total variation to
\begin{equation}\label{eq:full-bracket-reconstruction}
 B_i^\varphi(t)=\int_0^t\|G_i(u(s))^*\varphi\|_{U_i}^2\,\dd s.
\end{equation}
In addition, the Kunita--Watanabe inequality and \eqref{eq:full-integral-L2-limit} imply
\[
 \E\operatorname{Var}_{[0,T]}
 \langle M_i^\varphi,I_i^\varphi-I_i^{\varphi,N}\rangle
 \le
 \Bigl(\E\langle M_i^\varphi\rangle_T\Bigr)^{1/2}
 \Bigl(\E\langle I_i^\varphi-I_i^{\varphi,N}\rangle_T\Bigr)^{1/2}
 \longrightarrow0.
\]
Consequently
\[
 \langle M_i^\varphi,I_i^\varphi\rangle
 =\langle M_i^\varphi\rangle
 =\langle I_i^\varphi\rangle
 =B_i^\varphi.
\]
It follows that $M_i^\varphi-I_i^\varphi$ is a continuous local martingale with zero quadratic variation, and hence is identically zero.

The preceding identity holds on one probability-one event for every
$\varphi\in\mathcal D_0$.  Graph-norm approximation first extends it to all
$\varphi\in D(A)$.  To pass from $D(A)$ to $V$, define in $V^*$
\[
 \bar u_i(t)=u_{i,0}
 +\int_0^t\left(Au_i(s)+\sum_{j=1}^Kq_{ij}u_j(s)+\theta_i\right)\dd s
 +\int_0^tG_i(u(s))\,\dd W_i(s).
\]
Both sides have continuous $H_A^{-2}$-valued paths, and the identities on
$\mathcal D_0$ separate points of $H_A^{-2}$.  Hence
$u_i(t)=\bar u_i(t)$ for every $t\in[0,T]$ on the same event; at $t=0$
this also records explicitly the initial condition
$u_i(0)=u_{i,0}$ from \eqref{eq:limit-initial-condition}.  The variational
continuity theorem gives an $H$-continuous version of $\bar u_i$.
Replacing $u_i$ by this indistinguishable version, let
$\varphi_k\in D(A)$ tend to $\varphi\in V$ in the $V$-norm.  The heat term
converges by the variational pairing, the lower-order drift by its
$H$-bound, and the stochastic term by
\[
 \E\sup_{t\le T}\abs{I_i^{\varphi_k}(t)-I_i^\varphi(t)}^2
 \le4\E\int_0^T\|G_i(u(s))\|_{L_2(U_i,H)}^2
             \|\varphi_k-\varphi\|_H^2\,\dd s.
\]
Thus the variational identity holds for every $\varphi\in V$.

\emph{Step 4: continuity and estimates of the identified solution.}
The variational continuity theorem gives an $H^K$-continuous version; see \cite[Chapter~4]{PrevotRockner2007}.  Since the positive cone of $H^K$ is closed and $u$ is nonnegative for almost every time, this version is nonnegative at every time on a single probability-one event.  Instead of passing the supremum energy estimate through the path topology, apply \Cref{prop:energy-estimate} directly to the identified weak solution.  Likewise, \Cref{thm:power-main} was proved directly for every weak solution and gives the power entropy without a limiting lower-semicontinuity shortcut.  Finally, \Cref{prop:spatial-continuity} gives a continuous $C(\overline I)^K$-valued version.  This proves (ii)--(iv).
\end{proof}

\begin{lemma}[Diagonal property of joint limits]\label{lem:diagonal}
Let $(n_j)$ and $(m_j)$ be arbitrary subsequences.  Every joint subsequential limit of
\[
 (u^{n_j},u^{m_j},W)
\]
on $\cX_T\times\cX_T\times\cW_T$ is concentrated on
$\{(z,z,w):z\in\cX_T,\ w\in\cW_T\}$.
\end{lemma}

\begin{proof}
By \Cref{prop:tightness}, the joint laws are tight.  By the joint form of \Cref{prop:limit-identification}, the two limiting state coordinates are WV solutions on the same canonical stochastic basis, with the same initial state and the same limiting cylindrical Wiener process.  The Hellinger pathwise uniqueness theorem, which is stated in precisely this weak class, implies that the two state coordinates are indistinguishable.
\end{proof}

\begin{proof}[Proof of \Cref{thm:main}]
By \Cref{lem:diagonal} and the Polish-space diagonal criterion in \cite[Lemma~1.1]{GyongyKrylov2022}, the full sequence $u^n$ converges in probability in $\cX_T$ to a random variable $u$ on the original probability space.

\emph{Adaptedness to the prescribed noise.}
For every $t$, evaluation from $\cX_T$ to $H_A^{-2}(I)^K$ is continuous,
so $u^n(t)\to u(t)$ in probability in $H_A^{-2}(I)^K$.  Each $u^n(t)$
is measurable with respect to the completed natural $\sigma$-field
$\cF_t^W$; measurability is preserved under convergence in probability on
the completed $\sigma$-field.  Thus $u(t)$ is $\cF_t^W$-measurable.  Since
$u$ has continuous $H_A^{-2}$-valued paths, it has an
$(\cF_t^W)$-progressively measurable version.

\emph{Passage to the equation on the original space.}
Choose a subsequence converging almost surely in $\cX_T$.  The uniform
variational bounds give a further subsequence converging weakly in
$L^2(\Omega\times(0,T);V^K)$.  Convergence in probability in
$L^2(0,T;H^K)$ and the higher uniform energy moments imply, by Vitali's
theorem,
\[
 u^n\longrightarrow u
 \quad\text{strongly in }L^2(\Omega\times(0,T);H^K).
\]
It follows that the weak $V^K$-limit is $u$ and
$Au_i^n\rightharpoonup Au_i$ in
$L^2(\Omega\times(0,T);V^*)$.  Also, continuity of evaluation at zero and
$u_0^n\to u_0$ give
\begin{equation}\label{eq:initial-original}
 u(0)=u_0\qquad\text{in }H_A^{-2}(I)^K\quad\text{almost surely}.
\end{equation}

The $H_A^{-2}$-progressively measurable process $u$ is $H^K$-valued for
$\dd\Prob\,\dd t$-almost every $(\omega,t)$.  As in Step~3 above,
Lusin--Souslin for $H^K\hookrightarrow H_A^{-2}(I)^K$, followed by a
predictable modification on a product-null set, gives an $H^K$-valued
predictable representative.  Hence $G_i(u)$ is predictable and the
stochastic integral below is well defined.  By
\Cref{lem:coefficient-convergence},
\begin{equation}\label{eq:coeff-original}
 \E\int_0^T
 \|G_i^n(u^n(s))-G_i(u(s))\|_{L_2(U_i,H)}^2\,\dd s
 \longrightarrow0.
\end{equation}
The Hilbert-space BDG inequality consequently gives
\begin{equation}\label{eq:stoch-original}
 \E\sup_{t\le T}\left\|
 \int_0^t\bigl[G_i^n(u^n(s))-G_i(u(s))\bigr]\,\dd W_i(s)
 \right\|_H^2\longrightarrow0.
\end{equation}

Fix the countable graph-norm dense set
$\mathcal D_0\subset D(A)$ used above.  For $\varphi\in\mathcal D_0$,
self-adjointness of $A$ gives
\[
 \int_0^t\ip{Au_i^n(s)-Au_i(s)}{\varphi}_{V^*,V}\,\dd s
 =\int_0^t\ip{u_i^n(s)-u_i(s)}{A\varphi}\,\dd s,
\]
whose supremum over $t\le T$ tends to zero in $L^2(\Omega)$.  The
lower-order drift converges in $L^2(\Omega\times(0,T);H)$, the endpoint
terms converge in $C([0,T];H_A^{-2})$, and
\eqref{eq:stoch-original} treats the stochastic term.  After one further
subsequence, these convergences hold on a common probability-one event for
every $\varphi\in\mathcal D_0$, every component, and every rational
$t\in[0,T]$.  Passing to the regularized variational identities therefore
gives the limiting scalar identity at all such rational times.

To close the argument for all times and all variational test functions,
set in $V^*$
\[
 \bar u_i(t)=u_{i,0}
 +\int_0^t\left(Au_i(s)+\sum_{j=1}^Kq_{ij}u_j(s)+\theta_i\right)\dd s
 +\int_0^tG_i(u(s))\,\dd W_i(s).
\]
The path $\bar u_i$ is continuous in $V^*$, hence in $H_A^{-2}$, while
$u_i$ is continuous in $H_A^{-2}$.  The rational-time identities, path
continuity, and the fact that $\mathcal D_0$ separates points of
$H_A^{-2}$ show that $u_i(t)=\bar u_i(t)$ for every $t$ on the same event.
Density of $D(A)$ in $V$, continuity of the $V^*$--$V$ pairing, and the
BDG estimate for the stochastic pairing then extend the identity from
$\mathcal D_0$ to every $\varphi\in V$, without an uncountable
intersection of exceptional events.  This proves \eqref{eq:variational}
with the original $W$, including the initial condition
\eqref{eq:initial-original}.

The variational continuity theorem and \Cref{prop:spatial-continuity} give
the $H^K$- and $C(\overline I)^K$-continuous versions.  They are
indistinguishable from the adapted $H_A^{-2}$-continuous version, so the
resulting solution remains adapted and $G(u)$ remains predictable.  Thus
$u$ is a probabilistically strong solution in the sense of
\Cref{def:strong-solution}.

The energy estimate follows by applying \Cref{prop:energy-estimate} directly to $u$, and the mild formulation follows by semigroup duality.  Pathwise uniqueness is \Cref{thm:stability}, so the constructed process is the unique strong solution driven by the prescribed $W$.  Since $T$ is arbitrary, the solution is global.

To make the dependence on the horizon explicit, perform the construction for integer horizons and denote the resulting solution on $[0,N]$ by $u^{(N)}$.  The restriction of $u^{(N+1)}$ to $[0,N]$ and $u^{(N)}$ are strong solutions with the same initial state and the same prescribed Brownian family.  Pathwise uniqueness gives
$u^{(N+1)}|_{[0,N]}=u^{(N)}$ almost surely.  Intersecting these countably many probability-one events defines a single adapted process $u(t)=u^{(N)}(t)$ for $t\le N$, independent of the chosen integer $N$.  This process is the global strong solution.
\end{proof}

\appendix

\section{Factorization, compactness, and variational It\^o formulas}\label{app:compactness}

\begin{lemma}[Stochastic-convolution factorization]\label{lem:factorization}
Let $\alpha\in(0,1)$ and $q>2/(1-\alpha)$.  Let $B$ be progressively measurable with values in $L_2(U,H)$ and suppose
\begin{equation}\label{eq:factorization-assumption}
 \E\int_0^T\|B(s)\|_{L_2(U,H)}^q\,\dd s<\infty.
\end{equation}
Then the stochastic convolution
\[
 S(t)=\int_0^tP_{t-s}B(s)\,\dd W(s)
\]
has a continuous $H^\alpha(I)$-valued version and
\begin{equation}\label{eq:factorization-bound}
 \E\|S\|_{C([0,T];H^\alpha)}^q
 \le C_{\alpha,q,T}\E\int_0^T\|B(s)\|_{L_2(U,H)}^q\,\dd s.
\end{equation}
\end{lemma}

\begin{proof}
The analytic Neumann heat semigroup satisfies
\[
 \|P_rB\|_{L_2(U,H^\alpha)}^2
 \le C_\alpha(1+r^{-\alpha})\|B\|_{L_2(U,H)}^2.
\]
The restriction $q>2/(1-\alpha)$ is exactly what permits a factorization exponent $a$ satisfying
\begin{equation}\label{eq:factorization-exponent}
 \frac1q<a<\frac{1-\alpha}{2}.
\end{equation}
Define
\[
 Y_a(s)=\int_0^s(s-r)^{-a}P_{s-r}B(r)\,\dd W(r).
\]

Put $k(r)=r^{-2a}(1+r^{-\alpha})\one_{(0,T)}(r)$.  Since
$2a+\alpha<1$, one has $k\in L^1(0,T)$.  The BDG inequality followed by
Young's convolution inequality in time gives
\begin{align}
 \E\int_0^T\|Y_a(s)\|_{H^\alpha}^q\,\dd s
 &\le C_q\E\int_0^T
 \left[\int_0^s k(s-r)\|B(r)\|_{L_2(U,H)}^2\,\dd r\right]^{q/2}\dd s
 \notag\\
 &\le C_{\alpha,a,q,T}\E\int_0^T
 \|B(r)\|_{L_2(U,H)}^q\,\dd r.
 \label{eq:factorization-Ya-bound}
\end{align}
Thus $Y_a\in L^q(\Omega\times(0,T);H^\alpha)$ with the required
quantitative bound.  The factorization identity (cf. \cite[Theorem~5.10]{DaPratoZabczyk2014})
\[
 S(t)=\frac{\sin(\pi a)}{\pi}
 \int_0^t(t-s)^{a-1}P_{t-s}Y_a(s)\,\dd s
\]
defines a bounded deterministic convolution map
\[
 \mathcal R_a:L^q(0,T;H^\alpha)\longrightarrow
 C([0,T];H^\alpha),
 \qquad
 (\mathcal R_ay)(t)=\int_0^t(t-s)^{a-1}P_{t-s}y(s)\,\dd s.
\]
Indeed, H\"older's inequality (equivalently, Young's convolution
inequality) gives
\[
 \sup_{t\le T}\|\mathcal R_ay(t)\|_{H^\alpha}
 \le C_{a,q,T}
 \|r^{a-1}\|_{L^{q'}(0,T)}\|y\|_{L^q(0,T;H^\alpha)},
\]
and the kernel belongs to $L^{q'}(0,T)$ precisely because $a>1/q$.
Approximation by time-step functions, together with the strong
continuity of $(P_t)$, gives continuity of $\mathcal R_ay$ in
$H^\alpha$.  Applying this estimate to $Y_a$, and then the BDG estimate
above, yields \eqref{eq:factorization-bound}.
\end{proof}

\begin{lemma}[Compactness in the chosen path space]\label{lem:compactness-appendix}
Let $\mathcal K$ be a family bounded in
\[
 L^2(0,T;V)\cap L^\infty(0,T;H)
\]
and equicontinuous in $H^{-2}$ with a common positive H\"older exponent.  Then $\mathcal K$ is relatively compact in
\[
 L^2(0,T;H)\cap C([0,T];H^{-2}).
\]
\end{lemma}

\begin{proof}
The time modulus in $H^{-2}$ and the bound in $L^2(0,T;V)$ give relative compactness in $L^2(0,T;H)$ by the compactness theorem of Simon \cite[Theorem~5]{Simon1987}, using $V\Subset H\hookrightarrow H^{-2}$.  For every fixed time, the $L^\infty(0,T;H)$ bound makes the set of values relatively compact in $H^{-2}$ because $H\Subset H^{-2}$ on the bounded interval.  The common H\"older modulus and the Arzel\`a--Ascoli theorem therefore give relative compactness in $C([0,T];H^{-2})$.  A diagonal subsequence converges in both topologies.  
\end{proof}

\begin{lemma}[Variational It\^o formula for the regularized entropies]\label{lem:variational-ito}
Let $U$ be a separable Hilbert space, let $W$ be a cylindrical Wiener process on $U$, and let $X$ be an adapted variational process satisfying
\begin{equation}\label{eq:variational-process}
 X(t)=X_0+\int_0^t Y(s)\,\dd s+\int_0^t Z(s)\,\dd W(s)
 \quad\text{in }V^*,
\end{equation}
where
\begin{align}
 X&\in L^2\bigl(\Omega;C([0,T];H)\bigr)
     \cap L^2\bigl(\Omega\times(0,T);V\bigr),\label{eq:X-class}\\
 Y&\in L^2\bigl(\Omega\times(0,T);V^*\bigr),
 \qquad
 Z\in L^2\bigl(\Omega\times(0,T);L_2(U,H)\bigr).\label{eq:YZ-class}
\end{align}

Let $\tau\le T$ be a stopping time.  Assume either that
$f\in C^2(\R)$ has bounded first and second derivatives and uniformly
continuous second derivative, or that the following local-extension
hypothesis holds: on the stochastic interval $[0,\tau]$, $X$ admits a
jointly measurable $C(\overline I)$-valued version whose pointwise spatial
range is almost surely contained in a fixed deterministic compact interval
$J$, and $f|_J$ agrees with the restriction of a global $C^2(\R)$ function
having those derivative bounds.  

Define
\[
 F(w)=\int_I f(w(x))\,\dd x.
\]
Then, almost surely, the following identity holds simultaneously
for every $t\in[0,T]$:
\begin{align}
 F(X(t\wedge\tau))
 &=F(X_0)
   +\int_0^{t\wedge\tau}
      \ip{Y(s)}{f'(X(s))}_{V^*,V}\,\dd s\notag\\
 &\quad+\frac12\int_0^{t\wedge\tau}
      \sum_{k=1}^{\infty}\int_I
      f''(X(s,x))\abs{Z(s)e_k(x)}^2\,\dd x\,\dd s\notag\\
 &\quad+\int_0^{t\wedge\tau}
      \ip{Z(s)^*f'(X(s))}{\dd W(s)}_U,
 \label{eq:variational-ito-scalar}
\end{align}
where $(e_k)_{k\ge1}$ is any ONB of $U$.  The series in the trace is absolutely integrable and is independent of the chosen basis.
\end{lemma}
\begin{proof}
We divide the argument into seven steps.

\medskip
\noindent\text{Step 1: reduction to globally bounded derivatives.}
If $f$ already has the required global bounds, set $\widehat f=f$.

Otherwise, use the deterministic compact interval $J$ containing the
spatial range of the selected $C(\overline I)$-valued version of $X$
on $[0,\tau]$, and choose $\widehat f\in C^2(\R)$ that agrees with $f$
on $J$, has bounded first and second derivatives, and has uniformly
continuous second derivative.  Replacing $f$ by $\widehat f$ does not
change any term in \eqref{eq:variational-ito-scalar} before $\tau$.

We may therefore assume throughout the proof that
\begin{equation}\label{eq:bounded-f}
 \norm{f'}_{L^\infty(\R)}+\norm{f''}_{L^\infty(\R)}<\infty,
 \qquad f''\text{ is uniformly continuous.}
\end{equation}
In particular, $f$ has at most linear growth.

For $w,h,k\in H$, the first and second directional variations are
\begin{align}
 DF(w)h&=\int_I f'(w)h\,\dd x,\label{eq:DF}\\
 D^2F(w)[h,k]&=\int_I f''(w)hk\,\dd x.\label{eq:D2F}
\end{align}
The first variation is a bounded linear functional, and the second variation is a bounded bilinear form because
\[
 \abs{D^2F(w)[h,k]}
 \le \norm{f''}_\infty\norm{h}_H\norm{k}_H.
\]

If $w\in V$, then $f'(w)\in V$ and the Sobolev chain rule gives
\begin{equation}\label{eq:sobolev-chain}
 \partial_x f'(w)=f''(w)\partial_xw
 \quad\text{almost everywhere on }I.
\end{equation}
Consequently,
\begin{equation}\label{eq:Nemytskii-V-bound}
 \norm{f'(w)}_V^2
 \le C_f\bigl(1+\norm{w}_V^2\bigr).
\end{equation}
This verifies that the variational pairing in \eqref{eq:variational-ito-scalar} is meaningful.

\medskip
\noindent\text{Step 2: Neumann spectral projections.}
Let $(\varphi_j)_{j\ge0}$ be the real Neumann eigenbasis of $H$,
\[
 -\Delta_N\varphi_j=\mu_j\varphi_j,
 \qquad 0=\mu_0\le\mu_1\le\cdots,
\]
and let $P_N$ be the orthogonal projection onto
$H_N=\operatorname{span}\{\varphi_0,\ldots,\varphi_N\}$.  The operators $P_N$ are contractions on both $H$ and $V$, and
\begin{equation}\label{eq:PN-convergence}
 P_Nv\longrightarrow v\quad\text{strongly in }V
 \quad\text{for every }v\in V.
\end{equation}
For $y\in V^*$, define $P_Ny\in H_N$ by
\[
 \ip{P_Ny}{\varphi_j}_H=\ip{y}{\varphi_j}_{V^*,V},
 \qquad 0\le j\le N.
\]
Set
\[
 X^N=P_NX,
 \qquad Y^N=P_NY,
 \qquad Z^N=P_NZ.
\]
Taking the dual pairing of \eqref{eq:variational-process} with each $\varphi_j$ shows that $X^N$ is the $H_N$-valued continuous semimartingale
\begin{equation}\label{eq:projected-SDE}
 X^N(t)=P_NX_0+\int_0^tY^N(s)\,\dd s
                   +\int_0^tZ^N(s)\,\dd W(s).
\end{equation}

\medskip
\noindent\text{Step 3: finite-dimensional It\^o formula.}
Restrict $F$ to $H_N$.  Since $H_N$ is finite dimensional, the classical It\^o formula applied to \eqref{eq:projected-SDE}, stopped at $\tau$, gives
\begin{align}
 F(X^N(t\wedge\tau))
 &=F(P_NX_0)
   +\int_0^{t\wedge\tau}
      \ip{Y(s)}{P_Nf'(X^N(s))}_{V^*,V}\,\dd s\notag\\
 &\quad+\frac12\int_0^{t\wedge\tau}
      \sum_{k=1}^{\infty}\int_I
      f''(X^N(s,x))
      \abs{P_NZ(s)e_k(x)}^2\,\dd x\,\dd s\notag\\
 &\quad+M^N(t\wedge\tau),
 \label{eq:finite-ito}
\end{align}
where
\begin{equation}\label{eq:MN}
 M^N(t)=\int_0^t
 \ip{Z(s)^*P_Nf'(X^N(s))}{\dd W(s)}_U.
\end{equation}
The appearance of $P_Nf'(X^N)$ in the drift and martingale terms follows from \eqref{eq:DF}.  The trace term follows from \eqref{eq:D2F}; explicitly,
\[
 \operatorname{Tr}_U
 D^2F(X^N)[P_NZ,P_NZ]
 =\sum_{k=1}^{\infty}\int_I
 f''(X^N)\abs{P_NZe_k}^2\,\dd x.
\]
Absolute convergence follows from
\begin{equation}\label{eq:trace-bound-N}
 \sum_{k=1}^{\infty}
 \int_I\abs{f''(X^N)}\abs{P_NZe_k}^2\,\dd x
 \le\norm{f''}_\infty\norm{P_NZ}_{L_2(U,H)}^2.
\end{equation}
Equivalently, this series is the trace of the trace-class operator
$Z^*P_NM_{f''(X^N)}P_NZ$ on $U$, where $M_{f''(X^N)}$ denotes multiplication by $f''(X^N)$.  Hence its value is independent of the ONB $(e_k)$.

\medskip
\noindent\text{Step 4: convergence of the Nemytskii gradients in $V$.}
The spectral projections satisfy
\begin{equation}\label{eq:XN-strong-V}
 X^N\longrightarrow X
 \quad\text{in }L^2(\Omega\times(0,T);V).
\end{equation}
Indeed, this follows from \eqref{eq:PN-convergence}, the contraction property, and dominated convergence.  We claim that
\begin{equation}\label{eq:fprime-convergence}
 P_Nf'(X^N)\longrightarrow f'(X)
 \quad\text{in }L^2(\Omega\times(0,T);V).
\end{equation}
For the $H$ part, boundedness of $f''$ yields the Lipschitz estimate
\[
 \norm{f'(X^N)-f'(X)}_H
 \le\norm{f''}_\infty\norm{X^N-X}_H.
\]
For the derivatives, use \eqref{eq:sobolev-chain}:
\begin{align*}
 &\partial_xf'(X^N)-\partial_xf'(X)\\
 &\qquad=f''(X^N)(\partial_xX^N-\partial_xX)
       +\bigl(f''(X^N)-f''(X)\bigr)\partial_xX.
\end{align*}
The first term tends to zero in $L^2(\Omega\times(0,T)\times I)$ by \eqref{eq:XN-strong-V}.  After passing to an almost-everywhere convergent subsequence, the second term tends to zero by dominated convergence, because $f''$ is bounded and continuous and $\partial_xX\in L^2$.  Since every subsequence has a further subsequence with this limit, convergence holds for the full sequence.  Finally,
\begin{align*}
 \norm{P_Nf'(X^N)-f'(X)}_V
 &\le\norm{P_N(f'(X^N)-f'(X))}_V
      +\norm{P_Nf'(X)-f'(X)}_V,
\end{align*}
which proves \eqref{eq:fprime-convergence}.

Cauchy--Schwarz in $L^2(\Omega\times(0,T))$ gives the stronger,
uniform-in-time convergence
\begin{align}
 &\E\sup_{t\le T}\left|
 \int_0^{t\wedge\tau}
 \ip{Y(s)}{P_Nf'(X^N(s))-f'(X(s))}_{V^*,V}\,\dd s
 \right|\notag\\
 &\qquad\le
 \left(\E\int_0^T\|Y(s)\|_{V^*}^2\,\dd s\right)^{1/2}
 \left(\E\int_0^T
 \|P_Nf'(X^N(s))-f'(X(s))\|_V^2\,\dd s\right)^{1/2}
 \longrightarrow0.
 \label{eq:first-order-limit}
\end{align}

\medskip
\noindent\text{Step 5: convergence of the Hilbert--Schmidt trace.}
Let $(e_k)$ be an ONB of $U$ and write
\[
 \Gamma_N(s,x)=\sum_{k=1}^{\infty}\abs{P_NZ(s)e_k(x)}^2,
 \qquad
 \Gamma(s,x)=\sum_{k=1}^{\infty}\abs{Z(s)e_k(x)}^2.
\]
Both functions belong to $L^1(I)$ for almost every $(s,\omega)$, and
\[
 \int_I\Gamma_N\,\dd x=\norm{P_NZ}_{L_2(U,H)}^2,
 \qquad
 \int_I\Gamma\,\dd x=\norm{Z}_{L_2(U,H)}^2.
\]
Furthermore,
\begin{align}
 \int_I\abs{\Gamma_N-\Gamma}\,\dd x
 &\le
 \norm{P_NZ-Z}_{L_2(U,H)}
 \bigl(\norm{P_NZ}_{L_2(U,H)}+\norm{Z}_{L_2(U,H)}\bigr).
 \label{eq:Gamma-convergence}
\end{align}
Indeed, apply $|a^2-b^2|\le|a-b|(|a|+|b|)$ to every column and then Cauchy--Schwarz in the mode index and in $x$.  Since $P_NZ\to Z$ in $L^2(\Omega\times(0,T);L_2(U,H))$, the right-hand side of \eqref{eq:Gamma-convergence} tends to zero in $L^1(\Omega\times(0,T))$.

We split the trace difference as
\begin{align*}
 &\int_I f''(X^N)\Gamma_N\,\dd x
   -\int_If''(X)\Gamma\,\dd x\\
 &\quad=\int_I f''(X^N)(\Gamma_N-\Gamma)\,\dd x
       +\int_I\bigl(f''(X^N)-f''(X)\bigr)\Gamma\,\dd x.
\end{align*}
The first term tends to zero in $L^1(\Omega\times(0,T))$ by
\eqref{eq:Gamma-convergence}.

For the second term, regard
$\Gamma\,\dd\Prob\,\dd s\,\dd x$ as a finite measure.  The convergence
$X^N\to X$ in $L^2(\Omega\times(0,T)\times I)$ implies convergence in
measure also for this weighted finite measure.  Uniform continuity and
boundedness of $f''$ therefore give
\[
 \E\int_0^T\int_I
 |f''(X^N)-f''(X)|\Gamma\,\dd x\,\dd s\longrightarrow0.
\]
Consequently, if
\[
 \mathcal T_N(s)=\sum_{k=1}^{\infty}\int_I
 f''(X^N(s))|P_NZ(s)e_k|^2\,\dd x,
 \qquad
 \mathcal T(s)=\sum_{k=1}^{\infty}\int_I
 f''(X(s))|Z(s)e_k|^2\,\dd x,
\]
then
\begin{align}
 \E\sup_{t\le T}\left|
 \int_0^{t\wedge\tau}\bigl(\mathcal T_N(s)-\mathcal T(s)\bigr)\,\dd s
 \right|
 &\le \E\int_0^T|\mathcal T_N(s)-\mathcal T(s)|\,\dd s
 \longrightarrow0.
 \label{eq:trace-limit}
\end{align}

\medskip
\noindent\text{Step 6: convergence of the stochastic integrals.}
Set
\[
 \Psi_N=Z^*P_Nf'(X^N),
 \qquad
 \Psi=Z^*f'(X).
\]
Then, pointwise,
\begin{equation}\label{eq:psi-bound}
 \norm{\Psi_N-\Psi}_U^2
 \le\norm{Z}_{L_2(U,H)}^2
      \norm{P_Nf'(X^N)-f'(X)}_H^2.
\end{equation}

By \eqref{eq:fprime-convergence}, the second factor tends to zero in
measure and is bounded by a deterministic constant depending only on
$f$ and $|I|$.  Apply bounded convergence with respect to the finite
measure
$\|Z\|_{L_2(U,H)}^2\,\dd\Prob\,\dd s$ to obtain

\begin{equation}\label{eq:psi-convergence}
 \E\int_0^{T\wedge\tau}\norm{\Psi_N-\Psi}_U^2\,\dd s
 \longrightarrow0.
\end{equation}
The Burkholder--Davis--Gundy inequality therefore yields
\begin{equation}\label{eq:martingale-convergence}
 \E\sup_{t\le T}
 \abs{M^N(t\wedge\tau)-M(t\wedge\tau)}^2
 \longrightarrow0,
\end{equation}
where
\[
 M(t)=\int_0^t\ip{Z(s)^*f'(X(s))}{\dd W(s)}_U.
\]
In particular, the stochastic integral in the limiting formula is well defined.  Its quadratic variation satisfies
\[
 \E\langle M\rangle_{T\wedge\tau}
 =\E\int_0^{T\wedge\tau}\norm{Z^*f'(X)}_U^2\,\dd s
 \le |I|\norm{f'}_\infty^2
     \E\int_0^T\norm{Z}_{L_2(U,H)}^2\,\dd s<\infty,
\]
so $M(\cdot\wedge\tau)$ is a true square-integrable martingale.

\medskip
\noindent\text{Step 7: convergence of the endpoint terms.}

For almost every $\omega$, the image
\[
 \mathcal K_\omega=\{X(t,\omega):0\le t\le T\}
\]
of the continuous $H$-valued path is compact in $H$.  The projections
$P_N$ are contractions and converge strongly to the identity; hence
their convergence is uniform on compact subsets of $H$.  Consequently,
\begin{equation}\label{eq:PN-uniform-path}
 \sup_{t\le T}\|P_NX(t)-X(t)\|_H\longrightarrow0
 \quad\text{almost surely}.
\end{equation}

Since $f'$ is bounded,
\begin{equation}\label{eq:F-Lipschitz}
 \abs{F(v)-F(w)}
 \le |I|^{1/2}\norm{f'}_\infty\norm{v-w}_H.
\end{equation}

It follows from \eqref{eq:PN-uniform-path} that
\[
 \sup_{t\le T}
 |F(X^N(t\wedge\tau))-F(X(t\wedge\tau))|\longrightarrow0
 \quad\text{almost surely},
 \qquad
 F(P_NX_0)\longrightarrow F(X_0).
\]
The convergences \eqref{eq:first-order-limit},
\eqref{eq:trace-limit}, and \eqref{eq:martingale-convergence} take
place in the supremum norm in time (in $L^1$, $L^1$, and $L^2$,
respectively).  Choose one common subsequence along which all three
convergences hold uniformly in $t$ almost surely.  For each $N$, all
terms in \eqref{eq:finite-ito} have continuous paths, so the
finite-dimensional identity holds for every $t\in[0,T]$ on one event;
intersect these countably many events.  Passing to the limit along the
common subsequence on that event proves
\eqref{eq:variational-ito-scalar} simultaneously for every
$t\in[0,T]$.
\end{proof}

The proof of the following two-component version is similar, so we omit it.

\begin{lemma}
Let $X^a$, $Y^a$, and $Z^a$, $a=1,2$, satisfy
\begin{equation}\label{eq:two-component-process}
 X^a(t)=X_0^a+\int_0^tY^a(s)\,\dd s
 +\int_0^tZ^a(s)\,\dd W(s)\quad\text{in }V^*,
\end{equation}
with each triple satisfying \eqref{eq:X-class}--\eqref{eq:YZ-class}; the same cylindrical Wiener process $W$ drives both components.

Assume either that $f\in C^2(\R^2)$ has bounded gradient, bounded uniformly
continuous Hessian, and at most linear growth, or that the following
local-extension hypothesis holds: $(X^1,X^2)$ admits, on $[0,\tau]$, a
jointly measurable $C(\overline I)^2$-valued version whose pointwise
spatial range is almost surely contained in a fixed deterministic compact
rectangle $J_1\times J_2$, and $f$ agrees on a neighborhood of this
rectangle with a global $C^2(\R^2)$ function having those bounds. 

Let
\[
 F(w,z)=\int_If(w(x),z(x))\,\dd x.
\]
Then, almost surely, the following identity holds simultaneously
for every $t\in[0,T]$:
\begin{align}\label{eq:variational-ito-product}
 &F(X^1(t\wedge\tau),X^2(t\wedge\tau))
 =F(X_0^1,X_0^2)\notag\\
 &\quad+\int_0^{t\wedge\tau}\Bigl[
 {}_{V^*}\!\ip{Y^1}{f_1(X^1,X^2)}_V
 +{}_{V^*}\!\ip{Y^2}{f_2(X^1,X^2)}_V\Bigr]\,\dd s\notag\\
 &\quad+\frac12\int_0^{t\wedge\tau}\sum_{k=1}^{\infty}\int_I
 \Bigl[f_{11}|Z^1e_k|^2+2f_{12}(Z^1e_k)(Z^2e_k)
 +f_{22}|Z^2e_k|^2\Bigr] \,\dd x\dd s\notag\\
 &\quad+\int_0^{t\wedge\tau}
 \ip{(Z^1)^*f_1(X^1,X^2)+(Z^2)^*f_2(X^1,X^2)}{\dd W(s)}_U.
\end{align}
\end{lemma}

\section*{Acknowledgements}
This work was supported by the National Key R\&D Program of China (2022YFA1006102),
and the National Natural Science Foundation of China (12471418, 12595294, 12231002), and the New Cornerstone
Science Foundation (NCI202501).

\end{document}